\documentclass[a4paper,10pt]{amsart}

\title{Prop profile of bi-Hamiltonian structures}
\author{Henrik Strohmayer}
\address{Department of Mathematics \\ Stockholm University \\ 106 91 Stockholm \\ Sweden}
\email{henriks@math.su.se}

\usepackage{amsmath,amssymb,amsthm,amscd,latexsym,stmaryrd}
\usepackage[all]{xy}

\allowdisplaybreaks

\newcommand{\N}{{\ensuremath{\mathbb N}}}
\newcommand{\Z}{{\ensuremath{\mathbb Z}}}

\newcommand{\R}{{\ensuremath{\mathbb R}}}
\newcommand{\K}{{\ensuremath{\mathbb K}}}

\newcommand{\Bcal}{\ensuremath{\mathcal{B}}}
\newcommand{\Ccal}{\ensuremath{\mathcal{C}}}

\newcommand{\Jcal}{\ensuremath{\mathcal{J}}}

\newcommand{\Mcal}{{\ensuremath{\mathcal{M}}}}

\newcommand{\Ocal}{\ensuremath{\mathcal{O}}}
\newcommand{\Pcal}{\ensuremath{\mathcal{P}}}
\newcommand{\Qcal}{\ensuremath{\mathcal{Q}}}

\newcommand{\Tcal}{\ensuremath{\mathcal{T}}}
\newcommand{\Ucal}{\ensuremath{\mathcal{U}}}

\newcommand{\gfrak}{{\ensuremath{\mathfrak{g}}}}
\newcommand{\Gfrak}{{\ensuremath{\mathfrak{G}}}}

\newcommand{\G}{\Gamma}

\newcommand{\Com}{\mathcal{C}om}
\newcommand{\Lie}{\mathcal{L}ie}

\newcommand{\Lietwo}{\lc{\Lie}} %lc=linearly compatible

\newcommand{\Lieone}{\Lie^1}
\newcommand{\Comone}{\Com^1}

\newcommand{\Colie}{\mathcal{C}o\mathcal{L}ie}
\newcommand{\Colietwo}{\lc{\Colie}}

\newcommand{\End}{\mathcal{E}nd}

\newcommand{\Liebi}{\Lie \Bcal i}
\newcommand{\Lieonebi}{\Lieone \Bcal i}
\newcommand{\Lietwoonebi}{\Lieone_2 \Bcal i}

\newcommand{\Poisson}{\Lieonebi} % \mathcal{L}ie_1 \mathcal{C}o\mathcal{L}ie
\newcommand{\Bipoisson}{\Lietwoonebi}

\newcommand{\Poissonalg}{Lie $1$-bialgebra}
\newcommand{\Poissonalgs}{Lie $1$-bialgebras}
\newcommand{\Bipoissonalg}{$\text{Lie}_2$ $1$-bialgebra}

\newcommand{\s}{\ensuremath{\mathbb{S}}}
\newcommand{\smodule}{\ensuremath{\s}\text{-module}}
\newcommand{\sbimodule}{\ensuremath{\s}\text{-bimodule}}

\newcommand{\Free}{\mathcal{F}}
\newcommand{\QO}[2]{\Free(#1)/ ( #2 )}

\newcommand{\p}{\partial}
\newcommand{\ddt}[1]{\frac{\p\hspace{8pt}}{\p t^{#1}}}

\newcommand{\ol}[1]{{\ensuremath{\overline{#1}}}}
\newcommand{\olp}{\ensuremath{\overline{p}}}
\newcommand{\olq}{\ensuremath{\overline{q}}}
\newcommand{\wtp}{\ensuremath{\widetilde{p}}}

\newcommand{\Qf}{\hat{Q}}
\newcommand{\Pf}{\hat{P}}
\newcommand{\Pfa}{\hat{P}_1}
\newcommand{\Pfb}{\hat{P}_2}

\newcommand{\bvfs}{\wedge^2\Tcal_V}
\newcommand{\pvfs}{\wedge^\bullet\Tcal_V}
\newcommand{\pvfsh}{\wedge^\bullet\Tcal_V\llbracket\hslash\rrbracket}

\newcommand{\lc}[1]{#1^2}

\newcommand{\cz}{\scriptscriptstyle{\lor}}
\newcommand{\antishriek}{\text{\normalfont{<}}}

\DeclareMathOperator{\coker}{coker}

\DeclareMathOperator{\Ho}{H}
\DeclareMathOperator{\Hom}{Hom}

\DeclareMathOperator{\sgn}{sgn}

\DeclareMathOperator{\inp}{in}
\DeclareMathOperator{\outp}{out}

\DeclareMathOperator{\Bij}{Bij}

\DeclareMathOperator{\internal}{int}
\DeclareMathOperator{\connecting}{con}
\DeclareMathOperator{\op}{op}
\DeclareMathOperator{\Id}{Id}
\DeclareMathOperator{\ELL}{L}
\DeclareMathOperator{\CBC}{\Omega}
\DeclareMathOperator{\BC}{B}
\DeclareMathOperator{\dioperad}{dioperad}
\DeclareMathOperator{\properad}{properad}
\DeclareMathOperator{\propp}{prop}
\DeclareMathOperator{\wheeledpropp}{wheeled prop}

\newtheorem*{lem*}{Lemma}
\newtheorem{prop}[subsubsection]{Proposition}
\newtheorem*{prop*}{Proposition}
\newtheorem{thm}[subsubsection]{Theorem}
\newtheorem*{thm*}{Theorem}
\newtheorem{cor}[subsubsection]{Corollary}

\newtheorem{thmalpha}{Theorem}

\theoremstyle{definition}
\newtheorem{rem}[subsubsection]{Remark}
\newtheorem*{rem*}{Remark}
\newtheorem*{rems*}{Remarks}

\newtheorem*{defn*}{Definition}
\newtheorem{ex}[subsubsection]{Example}
\newtheorem*{ex*}{Example}

\newcommand{\bbone}{1\hspace{-2.6pt}\mbox{\normalfont{l}}}
\newcommand{\half}{{\ensuremath{\frac{1}{2}}}}

\newcommand{\iso}{\cong}
\newcommand{\niso}{\ncong}
\newcommand{\isoto}{\stackrel{\vspace{10pt}\sim}{\to}}
\newcommand{\isomapsto}{\stackrel{\vspace{10pt}\sim}{\mapsto}}

\newcommand{\bhs}{\hspace{10pt}}
\newcommand{\shs}{\hspace{5pt}}

\newcommand{\leftsub}[2]{{\vphantom{#2}}_{#1}{#2}}

\newcommand{\twovertexgraph}{\ensuremath{
 \xygraph{
!{<0pt,0pt>;<10pt,0pt>:<0pt,8pt>::}
!{(1,-0.6)}*{v_2 }
!{(1,-4.6)}*{v_1 }
!{(-6,5)}*{\scriptstyle i}
!{(-3,5)}*{\scriptstyle i+1 }
!{(3,5)}*{\scriptstyle i+n_1-2 \hspace {3pt}}
!{(6,5)}*{\scriptstyle \hspace{9pt} i+n_1-1 }
!{(-6,4)}="a"
!{(-3,4)}="b"
!{(0,4)}*{... }="c"
!{(3,4)}="d"
!{(6,4)}="e"
!{(-12,1)}*{\scriptstyle 1 }
!{(-9,1)}*{\scriptstyle 2 }
!{(-3,1)}*{\scriptstyle i-1 }
!{(3,1)}*{\scriptstyle \hspace{6pt} i+n_1 }
!{(9,1)}*{\scriptstyle  n_1+n_2-2 \hspace{9pt}}
!{(12,1)}*{\scriptstyle \hspace{15pt} n_1+n_2-1 }
!{(-12,0)}="ff"
!{(-9,0)}="f"
!{(-6,0,)}*{... }="g"
!{(-3,0)}="h"
!{(0,0)}="i"
!{(3,0)}="j"
!{(6,0)}*{... }="k"
!{(9,0)}="l"
!{(12,0)}="ll"
!{(0,-4)}="m"
!{(0,-8)}="n"
"a"-"i"
"b"-"i"
"d"-"i"
"e"-"i"
"ff"-"m"
"f"-"m"
"h"-"m"
"i"-"m"
"j"-"m"
"l"-"m"
"ll"-"m"
"m"-"n"
}
}}

\newcommand{\GraphG}{\ensuremath{
 \xygraph{
!{<0pt,0pt>;<6pt,0pt>:<0pt,6pt>::}
!{(-3.5,4)}*{ v_1}
!{(3.5,0)}*{ v_2}
!{(-3.5,-4)}*{ v_3}
!{(-5,7)}="a"
!{(-2,7)}="b"
!{(1,7)}="c"
!{(-2,4)}="d"
!{(5,3)}="e"
!{(-5,1)}="f"
!{(2,0)}="g"
!{(5,-3)}="h"
!{(-2,-4)}="i"
!{(-5,-7)}="j"
!{(1,-7)}="k"
"a"-"d"
"b"-"d"
"c"-"d"
"d"-"f"
"d"-@/^3pt/"i"
"d"-@/_3pt/"i"
"d"-@/^3pt/"g"
"d"-@/_3pt/"g"
"e"-"g"
"g"-@/_3pt/"i"
"g"-@/^3pt/"i"
"g"-"h"
"i"-"j"
"i"-"k"
}
}}

\newcommand{\GraphHb}{\ensuremath{
 \xygraph{
!{<0pt,0pt>;<6pt,0pt>:<0pt,6pt>::}
!{(-3.5,4)}*{v_1}
!{(-3.5,-4)}*{v_3}
!{(-5,7)}="a"
!{(-2,7)}="b"
!{(1,7)}="c"
!{(-2,4)}="d"
!{(-5,1)}="f"
!{(0,1)}="ga"
!{(2,1)}="gb"
!{(0,-1)}="gc"
!{(2,-1)}="gd"
!{(5,-3)}="h"
!{(-2,-4)}="i"
!{(-5,-7)}="j"
!{(1,-7)}="k"
"a"-"d"
"b"-"d"
"c"-"d"
"d"-"f"
"d"-@/^3pt/"i"
"d"-@/_3pt/"i"
"d"-"ga"
"d"-"gb"
"gc"-"i"
"gd"-"i"
"i"-"j"
"i"-"k"
}
}}

\newcommand{\GraphGHb}{\ensuremath{
 \xygraph{
!{<0pt,0pt>;<6pt,0pt>:<0pt,6pt>::}
!{(-6,0)}*{v_{H}}
!{(5.5,0)}*{v_2}
!{(-7,3)}="a"
!{(-4,3)}="b"
!{(-1,3)}="c"
!{(7,3)}="d"
!{(-4,0)}="e"
!{(4,0)}="f"
!{(-7,-3)}="g"
!{(-4,-3)}="h"
!{(-1,-3)}="i"
!{(7,-3)}="j"
"a"-"e"
"b"-"e"
"c"-"e"
"d"-"f"
"e"-@/^8pt/"f"
"e"-@/^3pt/"f"
"e"-@/_3pt/"f"
"e"-@/_8pt/"f"
"e"-"g"
"e"-"h"
"e"-"i"
"f"-"j"
}
}}

\newcommand{\GraphGnumbered}{\ensuremath{
 \xygraph{
!{<0pt,0pt>;<6pt,0pt>:<0pt,6pt>::}
!{(-5,8)}*{\scriptstyle 1}
!{(-2,8)}*{\scriptstyle 3}
!{(1,8)}*{\scriptstyle 4}
!{(5,4)}*{\scriptstyle 2}
!{(-5,0)}*{\scriptstyle 2}
!{(5,-4)}*{\scriptstyle 1}
!{(-5,-8)}*{\scriptstyle 3}
!{(1,-8)}*{\scriptstyle 4}
!{(-5,7)}="a"
!{(-2,7)}="b"
!{(1,7)}="c"
!{(-2,4)}="d"
!{(5,3)}="e"
!{(-5,1)}="f"
!{(2,0)}="g"
!{(5,-3)}="h"
!{(-2,-4)}="i"
!{(-5,-7)}="j"
!{(1,-7)}="k"
"a"-"d"
"b"-"d"
"c"-"d"
"d"-"f"
"d"-@/^3pt/"i"
"d"-@/_3pt/"i"
"d"-@/^3pt/"g"
"d"-@/_3pt/"g"
"e"-"g"
"g"-@/_3pt/"i"
"g"-@/^3pt/"i"
"g"-"h"
"i"-"j"
"i"-"k"
}
}}

\newcommand{\GraphsigmaGtau}{\ensuremath{
 \xygraph{
!{<0pt,0pt>;<6pt,0pt>:<0pt,6pt>::}
!{(-5,8)}*{\scriptstyle \tau^{-1}(1)\hspace{22pt}}
!{(-2,8)}*{\scriptstyle \tau^{-1}(3)}
!{(1,8)}*{\scriptstyle \hspace{22pt}\tau^{-1}(4)}
!{(5,4)}*{\scriptstyle \tau^{-1}(2)}
!{(-5,0)}*{\scriptstyle \sigma(2)}
!{(5,-4)}*{\scriptstyle \sigma(1)}
!{(-5,-8)}*{\scriptstyle \sigma(3)}
!{(1,-8)}*{\scriptstyle \sigma(4)}
!{(-5,7)}="a"
!{(-2,7)}="b"
!{(1,7)}="c"
!{(-2,4)}="d"
!{(5,3)}="e"
!{(-5,1)}="f"
!{(2,0)}="g"
!{(5,-3)}="h"
!{(-2,-4)}="i"
!{(-5,-7)}="j"
!{(1,-7)}="k"
"a"-"d"
"b"-"d"
"c"-"d"
"d"-"f"
"d"-@/^3pt/"i"
"d"-@/_3pt/"i"
"d"-@/^3pt/"g"
"d"-@/_3pt/"g"
"e"-"g"
"g"-@/_3pt/"i"
"g"-@/^3pt/"i"
"g"-"h"
"i"-"j"
"i"-"k"
}
}}

\newcommand{\econinout}{\ensuremath{
 \xygraph{
!{<0pt,0pt>;<10pt,0pt>:<0pt,10pt>::}
!{(-1,3)}*{ w_i}
!{(-1,-3)}*{w_j}
!{(0.7,2)}*{ e_i}="a"
!{(0.7,0)}*{ e}="b"
!{(0.7,-2)}*{ e_j}="c"
!{(1.5,2)}="aa"
!{(1.5,0.2)}="bbm"
!{(1.5,-0.2)}="bbn"
!{(1.5,-2)}="cc"
!{(-0.5,2)}="aaa"
!{(-0.5,0.2)}="bbbm"
!{(-0.5,-0.2)}="bbbn"
!{(-0.5,-2)}="ccc"
!{(0,3)}*{\bullet }="d"
!{(0,1.2)}="e"
!{(0,0.8)}="f"
!{(0,-0.8)}="g"
!{(0,-1.2)}="h"
!{(0,-3)}*{\bullet }="i"
"aa"-@/^3pt/@{->}"bbm"^{\connecting}
"bbn"-@/^3pt/@{->}"cc"^{\outp}
"ccc"-@/^3pt/@{->}"bbbn"^{\connecting}
"d"-"e"
"f"-"g"
"h"-"i"
}
}}
\newcommand{\XX}[5]{\ensuremath{
 \xygraph{
!{<0pt,0pt>;<4pt,0pt>:<0pt,4pt>::}
!{(-3,3.5)}*{\scriptscriptstyle #1}
!{(1,3.5)}*{\scriptscriptstyle #2}
!{(-3,2.5)}="a"
!{(1,2.5)}="b"
!{(3,1.5)}*{\scriptscriptstyle #3}
!{(-1,0.5)}*{\scriptscriptstyle #4}="c"
!{(3,0.5)}="d"
!{(1,-1.5)}*{\scriptscriptstyle #5}="e"
!{(1,-3.5)}="f"
"a"-"c"
"b"-"c"
"c"-"e"
"d"-"e"
"e"-"f"
}
}}

\newcommand{\YY}[3]{\XX{#1}{#2}{#3}{{}}{{}}}

\newcommand{\XXb}[6]{\ensuremath{
 \xygraph{
!{<0pt,0pt>;<4pt,0pt>:<0pt,4pt>::}
!{(-3,3.5)}*{\scriptscriptstyle #1}
!{(1,3.5)}*{\scriptscriptstyle #2}
!{(-3,2.5)}="a"
!{(1,2.5)}="b"
!{(3,1.5)}*{\scriptscriptstyle #3}
!{(-1,0.5)}*{\scriptscriptstyle #4}="c"
!{(3,0.5)}="d"
!{(1,-1.5)}*{\scriptscriptstyle #5}="e"
!{(1,-3.5)}="f"
!{(1,-4.5)}*{\scriptscriptstyle #6}
"a"-"c"
"b"-"c"
"c"-"e"
"d"-"e"
"e"-"f"
}
}}

\newcommand{\YYb}[3]{\XXb{#1}{#2}{#3}{{}}{{}}{1}}

\newcommand{\coXcoX}[5]{\ensuremath{
 \xygraph{
!{<0pt,0pt>;<4pt,0pt>:<0pt,4pt>::}
!{(1,3.5)}="a"
!{(1,1.5)}*{\scriptscriptstyle #4}="b"
!{(-1,-0.5)}*{\scriptscriptstyle #5}="c"
!{(3,-0.5)}="d"
!{(3,-1.5)}*{\scriptscriptstyle #3}
!{(-3,-2.5)}="e"
!{(1,-2.5)}="f"
!{(-3,-3.5)}*{\scriptscriptstyle #1}
!{(1,-3.5)}*{\scriptscriptstyle #2}
"a"-"b"
"b"-"c"
"b"-"d"
"c"-"e"
"c"-"f"
}
}}

\newcommand{\coYcoY}[3]{\coXcoX{#1}{#2}{#3}{{}}{{}}}

\newcommand{\coWcoW}[3]{\coXcoX{#1}{#2}{#3}{\circ}{\circ}}
\newcommand{\coWcoB}[3]{\coXcoX{#1}{#2}{#3}{\circ}{\bullet}}

\newcommand{\coBcoW}[3]{\coXcoX{#1}{#2}{#3}{\bullet}{\circ}}
\newcommand{\coBcoB}[3]{\coXcoX{#1}{#2}{#3}{\bullet}{\bullet}}

\newcommand{\XcoX}[6]{\ensuremath{
 \xygraph{
!{<0pt,0pt>;<4pt,0pt>:<0pt,4pt>::}
!{(-2,4)}*{\scriptscriptstyle #1}
!{(2,4)}*{\scriptscriptstyle #2}
!{(-2,3)}="a"
!{(2,3)}="b"
!{(0,1)}*{\scriptscriptstyle #5}="c"
!{(0,-1)}*{\scriptscriptstyle #6}="d"
!{(-2,-3)}="e"
!{(2,-3)}="f"
!{(-2,-4)}*{\scriptscriptstyle #3}
!{(2,-4)}*{\scriptscriptstyle #4}
"a"-"c"
"b"-"c"
"c"-"d"
"d"-"e"
"d"-"f"
}
}}

\newcommand{\YcoY}[4]{\XcoX{#1}{#2}{#3}{#4}{{}}{{}}}
\newcommand{\YcoW}[4]{\XcoX{#1}{#2}{#3}{#4}{{}}{\circ}}
\newcommand{\YcoB}[4]{\XcoX{#1}{#2}{#3}{#4}{{}}{\bullet}}

\newcommand{\coXopX}[6]{\ensuremath{
 \xygraph{
!{<0pt,0pt>;<4pt,0pt>:<0pt,4pt>::}
!{(-1,4)}*{\scriptscriptstyle #1}
!{(-1,3)}="a"
!{(3,2)}*{\scriptscriptstyle #2}
!{(-1,1)}*{\scriptscriptstyle #5}="b"
!{(3,1)}="c"
!{(-3,-1)}="d"
!{(1,-1)}*{\scriptscriptstyle #6}="e"
!{(-3,-2)}*{\scriptscriptstyle #3}
!{(1,-3)}="f"
!{(1,-4)}*{\scriptscriptstyle #4}
"a"-"b"
"b"-"d"
"b"-"e"
"c"-"e"
"e"-"f"
}
}}

\newcommand{\coYopY}[4]{\coXopX{#1}{#2}{#3}{#4}{{}}{{}}}

\newcommand{\coWopY}[4]{\coXopX{#1}{#2}{#3}{#4}{\circ}{{}}}

\newcommand{\coBopY}[4]{\coXopX{#1}{#2}{#3}{#4}{\bullet}{{}}}

\newcommand{\Ysmall}{\ensuremath{
 \xygraph{
!{<0pt,0pt>;<2.5pt,0pt>:<0pt,-2.5pt>::}
!{(0,1.5)}="a"
!{(0,-0.5)}="b"
!{(-2,-2.5)}="c"
!{(2,-2.5)}="d"
"a"-"b"
"b"-"c"
"b"-"d"
}
}}

\newcommand{\Wsmall}{\ensuremath{
 \xygraph{
!{<0pt,0pt>;<2.5pt,0pt>:<0pt,-2.5pt>::}
!{(0,1.5)}="a"
!{(0,-0.5)}*{\scriptscriptstyle{\circ}}="b"
!{(-2,-2.5)}="c"
!{(2,-2.5)}="d"
"a"-"b"
"b"-"c"
"b"-"d"
}
}}

\newcommand{\Bsmall}{\ensuremath{
 \xygraph{
!{<0pt,0pt>;<2.5pt,0pt>:<0pt,-2.5pt>::}
!{(0,1.5)}="a"
!{(0,-0.5)}*{\scriptscriptstyle{\bullet}}="b"
!{(-2,-2.5)}="c"
!{(2,-2.5)}="d"
"a"-"b"
"b"-"c"
"b"-"d"
}
}}

\newcommand{\coYsmall}{\ensuremath{
 \xygraph{
!{<0pt,0pt>;<2.5pt,0pt>:<0pt,2.5pt>::}
!{(0,2.5)}="a"
!{(0,0.5)}="b"
!{(-2,-1.5)}="c"
!{(2,-1.5)}="d"
"a"-"b"
"b"-"c"
"b"-"d"
}
}}

\newcommand{\coWsmall}{\ensuremath{
 \xygraph{
!{<0pt,0pt>;<2.5pt,0pt>:<0pt,2.5pt>::}
!{(0,2.5)}="a"
!{(0,0.5)}*{\scriptscriptstyle{\circ}}="b"
!{(-2,-1.5)}="c"
!{(2,-1.5)}="d"
"a"-"b"
"b"-"c"
"b"-"d"
}
}}

\newcommand{\coBsmall}{\ensuremath{
 \xygraph{
!{<0pt,0pt>;<2.5pt,0pt>:<0pt,2.5pt>::}
!{(0,2.5)}="a"
!{(0,0.5)}*{\scriptscriptstyle{\bullet}}="b"
!{(-2,-1.5)}="c"
!{(2,-1.5)}="d"
"a"-"b"
"b"-"c"
"b"-"d"
}
}}

\newcommand{\YcoWcoBcorolla}{\ensuremath{
 \xygraph{
!{<0pt,0pt>;<4pt,0pt>:<0pt,4pt>::}
!{(-1,7)}="a"
!{(3,7)}="b"
!{(1,5)}="c"
!{(5,5)}="d"
!{(1.5,4.5)}*{.}="cea"
!{(2,4)*{.}}="ceb"
!{(2.5,3.5)}*{.}="cec"
!{(3,3)}="e"
!{(3,1)}*{\scriptscriptstyle{\circ}}="f"
!{(2.5,0.5)}*{.}="fga"
!{(2,0)}*{.}="fgb"
!{(1.5,-0.5)}*{.}="fgc"
!{(1,-1)}*{\scriptscriptstyle{\circ}}="g"
!{(5,-1)}="h"
!{(-1,-3)}*{\scriptscriptstyle{\bullet}}="i"
!{(-1.5,-3.5)}*{.}="ika"
!{(-2,-4)}*{.}="ikb"
!{(-2.5,-4.5)}*{.}="ikc"
!{(3,-3)}="j"
!{(-3,-5)}*{\scriptscriptstyle{\bullet}}="k"
!{(1,-5)}="l"
!{(-5,-7)}="m"
!{(-1,-7)}="n"
!{(-1,8)}*{\scriptscriptstyle{1}}="A"
!{(3,8)}*{\scriptscriptstyle{2}}="B"
!{(5,6)}*{\scriptscriptstyle{n}}="D"
!{(-5,-8)}*{\scriptscriptstyle{1}}="M"
!{(-1,-8)}*{\scriptscriptstyle{2}}="N"
!{(1,-6)}*{\scriptscriptstyle{i+1}}="L"
!{(3,-4)}*{\scriptscriptstyle{i+2}}="J"
!{(5,-2)}*{\scriptscriptstyle{m}}="H"
"a"-"c"
"b"-"c"
"d"-"e"
"e"-"f"
"f"-"h"
"g"-"i"
"g"-"j"
"i"-"l"
"k"-"m"
"k"-"n"
}
}}

\newcommand{\boxcorolla}[5]{\ensuremath{
 \xygraph{
!{<0pt,0pt>;<4.5pt,0pt>:<0pt,4.5pt>::}
!{(-4,4)}="a"
!{(-2,4)}="b"
!{(0,4)}="c"
!{(2,4)}="d"
!{(4,4)}="e"
!{(-2,1)}="f"
!{(-1,1)}="g"
!{(1,1)}="h"
!{(2,1)}="i"
!{(0,0)}*{\scriptscriptstyle #3}%="j"
!{(-2,-1)}="k"
!{(-1,-1)}="l"
!{(1,-1)}="m"
!{(2,-1)}="n"
!{(-4,-4)}="o"
!{(-2,-4)}="p"
!{(0,-4)}="q"
!{(2,-4)}="r"
!{(4,-4)}="s"
!{(-4.5,5)}*{\scriptscriptstyle #1}  %="A"
!{(-0.8,4)}*{.}                   %="C"
!{(0,4)}*{.}                   %="C"
!{(0.8,4)}*{.}                   %="C"
!{(4.5,5)}*{\scriptscriptstyle #2}   %="E"
!{(-4.5,-5)}*{\scriptscriptstyle #4} %="O"
!{(-0.8,-4)}*{.}                  %="Q"
!{(0,-4)}*{.}                  %="Q"
!{(0.8,-4)}*{.}                  %="Q"
!{(4.5,-5)}*{\scriptscriptstyle #5}  %="S"
"a"-"f"
"b"-"g"
"d"-"h"
"e"-"i"
"f"-"i"
"f"-"k"
"i"-"n"
"k"-"n"
"k"-"o"
"l"-"p"
"m"-"r"
"n"-"s"
}
}}

\newcommand{\mirroredsplitboxcorolla}[8]{\ensuremath{
 \xygraph{
!{<0pt,0pt>;<4.5pt,0pt>:<0pt,-4.5pt>::}
!{(0,8)}="a"
!{(2,8)}="b"
!{(4,8)}="c"
!{(6,8)}="d"
!{(8,8)}="e"
!{(2,5)}="f"
!{(3,5)}="g"
!{(5,5)}="h"
!{(6,5)}="i"
!{(4,4)}*{\scriptscriptstyle i_2}%="j"
!{(2,3)}="k"
!{(3,3)}="l"
!{(5,3)}="m"
!{(6,3)}="n"
!{(0,0)}="o"
!{(2,0)}="p"
!{(4,0)}="q"
!{(6,0)}="r"
!{(8,0)}="s"
!{(-0.5,9)}*{\scriptscriptstyle #1}  %="a"
!{(3.2,8)}*{.}                   %="c"
!{(4,8)}*{.}                   %="c"
!{(4.8,8)}*{.}                   %="c"
!{(8.5,9)}*{\scriptscriptstyle #2}   %="e"
!{(2.5,-1)}*{\scriptscriptstyle #3} %="p"
!{(3.2,0)}*{.}                   %="c"
!{(4,0)}*{.}                   %="c"
!{(4.8,0)}*{.}                   %="c"
!{(8.5,-1)}*{\scriptscriptstyle #4}  %="s"
!{(-8,0)}="A"
!{(-6,0)}="B"
!{(-4,0)}="C"
!{(-2,0)}="D"
!{(-6,-3)}="F"
!{(-5,-3)}="G"
!{(-3,-3)}="H"
!{(-2,-3)}="I"
!{(-4,-4)}*{\scriptscriptstyle i_1}%="j"%!{(-4,-4)}*{\scriptscriptstyle #9,#10}%="J"
!{(-6,-5)}="K"
!{(-5,-5)}="L"
!{(-3,-5)}="M"
!{(-2,-5)}="N"
!{(-8,-8)}="O"
!{(-6,-8)}="P"
!{(-4,-8)}="Q"
!{(-2,-8)}="R"
!{(0,-8)}="S"
!{(-8.5,1)}*{\scriptscriptstyle #5}  %="A"
!{(-4.8,0)}*{.}                   %="C"-1
!{(-4,0)}*{.}                   %="C"-1
!{(-3.2,0)}*{.}                   %="C"-1
!{(-1.6,1)}*{\scriptscriptstyle #6}   %="D"
!{(-8.5,-9)}*{\scriptscriptstyle #7} %="O"
!{(-4.8,-8)}*{.}                   %="C"-1
!{(-4,-8)}*{.}                   %="C"-1
!{(-3.2,-8)}*{.}                   %="C"-1
!{(0.5,-9)}*{\scriptscriptstyle #8}  %="S"
"a"-"f"
"b"-"g"
"d"-"h"
"e"-"i"
"f"-"i"
"f"-"k"
"i"-"n"
"k"-"n"
"k"-"I"
"l"-"p"
"m"-"r"
"n"-"s"
"A"-"F"
"B"-"G"
"D"-"H"
"F"-"I"
"F"-"K"
"I"-"N"
"K"-"N"
"K"-"O"
"L"-"P"
"M"-"R"
"N"-"S"
}
}}

\newcommand{\mirroredboxcorolladiffup}[5]{\ensuremath{
 \xygraph{
!{<0pt,0pt>;<4.5pt,0pt>:<0pt,-4.5pt>::}
!{(-4,4)}="a"
!{(-2,4)}="b"
!{(0,4)}="c"
!{(2,4)}="d"
!{(4,4)}="e"
!{(4,8)}="t"
!{(4,9)}*{\scriptscriptstyle \sigma(#2)}
!{(4,6)}*{\scriptstyle \times}="u"
!{(-2,1)}="f"
!{(-1,1)}="g"
!{(1,1)}="h"
!{(2,1)}="i"
!{(0,0)}*{\scriptscriptstyle #3}%="j"
!{(-2,-1)}="k"
!{(-1,-1)}="l"
!{(1,-1)}="m"
!{(2,-1)}="n"
!{(-4,-4)}="o"
!{(-2,-4)}="p"
!{(0,-4)}="q"
!{(2,-4)}="r"
!{(4,-4)}="s"
!{(-6.5,5)}*{\scriptscriptstyle \sigma(#1)}  %="A"
!{(-3.5,5)}*{\dots}                   %="C"
!{(1,5)}*{\scriptscriptstyle \sigma(#2-1)}   %="E"
!{(-4.5,-5)}*{\scriptscriptstyle #4} %="O"
!{(0,-5)}*{\dots}                  %="Q"
!{(4.5,-5)}*{\scriptscriptstyle #5}  %="S"
"a"-"f"
"b"-"g"
"d"-"h"
"e"-"i"
"f"-"i"
"f"-"k"
"i"-"n"
"k"-"n"
"k"-"o"
"l"-"p"
"m"-"r"
"n"-"s"
"t"-"u"
"u"-"e"
}
}}

\newcommand{\mirroredboxcorolladiffdown}[5]{\ensuremath{
 \xygraph{
!{<0pt,0pt>;<4.5pt,0pt>:<0pt,-4.5pt>::}
!{(-4,4)}="a"
!{(-2,4)}="b"
!{(0,4)}="c"
!{(2,4)}="d"
!{(4,4)}="e"
!{(-2,1)}="f"
!{(-1,1)}="g"
!{(1,1)}="h"
!{(2,1)}="i"
!{(0,0)}*{\scriptscriptstyle #3}%="j"
!{(-2,-1)}="k"
!{(-1,-1)}="l"
!{(1,-1)}="m"
!{(2,-1)}="n"
!{(-4,-4)}="o"
!{(-4,-6)}*{\scriptstyle \times}="t"
!{(-4,-8)}="u"
!{(-4,-9)}*{\scriptscriptstyle \tau(1)}
!{(-2,-4)}="p"
!{(0,-4)}="q"
!{(2,-4)}="r"
!{(4,-4)}="s"
!{(-4.5,5)}*{\scriptscriptstyle #1}  %="A"
!{(0,5)}*{\dots}                   %="C"
!{(4.5,5)}*{\scriptscriptstyle #2}   %="E"
!{(-2,-5)}*{\scriptscriptstyle \tau(#4)} %="O"
!{(1.5,-5)}*{\dots}                  %="Q"
!{(5.5,-5)}*{\scriptscriptstyle \tau(#5)}  %="S"
"a"-"f"
"b"-"g"
"d"-"h"
"e"-"i"
"f"-"i"
"f"-"k"
"i"-"n"
"k"-"n"
"k"-"o"
"l"-"p"
"m"-"r"
"n"-"s"
"o"-"t"
"t"-"u"
}
}}

\newcommand{\lietwoboxcorolla}[3]{\ensuremath{
 \xygraph{
!{<0pt,0pt>;<4.5pt,0pt>:<0pt,4.5pt>::}
!{(-4,4)}="a"
!{(-2,4)}="b"
!{(0,4)}="c"
!{(2,4)}="d"
!{(4,4)}="e"
!{(-2,1)}="f"
!{(-1,1)}="g"
!{(1,1)}="h"
!{(2,1)}="i"
!{(0,0)}*{\scriptscriptstyle #3}%="j"
!{(-2,-1)}="k"
!{(-1,-1)}="l"
!{(1,-1)}="m"
!{(2,-1)}="n"
!{(0,-1)}="o"
!{(0,-4)}="p"
!{(-4.5,5)}*{\scriptscriptstyle #1}  %="A"
!{(0,5)}*{\dots}                   %="C"
!{(4.5,5)}*{\scriptscriptstyle #2}   %="E"
"a"-"f"
"b"-"g"
"d"-"h"
"e"-"i"
"f"-"i"
"f"-"k"
"i"-"n"
"k"-"n"
"o"-"p"
}
}}

\newcommand{\lietwosplitboxcorolla}[4]{\ensuremath{
 \xygraph{
!{<0pt,0pt>;<-4.5pt,0pt>:<0pt,4.5pt>::}
!{(-2,8)}="a"
!{(0,8)}="b"
!{(2,8)}="c"
!{(4,8)}="d"
!{(6,8)}="e"
!{(0,5)}="f"
!{(1,5)}="g"
!{(3,5)}="h"
!{(4,5)}="i"
!{(2,4)}*{\scriptscriptstyle i_1}%="j"
!{(0,3)}="k"
!{(1,3)}="l"
!{(2,3)}="lm"
!{(3,3)}="m"
!{(4,3)}="n"
!{(-2,0)}="o"
!{(-2.5,9)}*{\scriptscriptstyle #2}  %="a"
!{(1.2,8)}*{.}                   %="c"
!{(2,8)}*{.}                   %="c"
!{(2.8,8)}*{.}                   %="c"
!{(6.5,9)}*{\scriptscriptstyle #1}   %="e"
!{(-8,0)}="A"
!{(-6,0)}="B"
!{(-4,0)}="C"
!{(-2,0)}="D"
!{(-6,-3)}="F"
!{(-5,-3)}="G"
!{(-3,-3)}="H"
!{(-2,-3)}="I"
!{(-4,-4)}*{\scriptscriptstyle i_2}%="j"%!{(-4,-4)}*{\scriptscriptstyle #9,#10}%="J"
!{(-6,-5)}="K"
!{(-5,-5)}="L"
!{(-4,-5)}="LM"
!{(-3,-5)}="M"
!{(-2,-5)}="N"
!{(-4,-8)}="Q"
!{(-9,1)}*{\scriptscriptstyle #4}  %="A"
!{(-4.8,0)}*{.}                   %="C"-1
!{(-4,0)}*{.}                   %="C"-1
!{(-3.2,0)}*{.}                   %="C"-1
!{(-2.5,1)}*{\scriptscriptstyle #3}   %="D"
"a"-"f"
"b"-"g"
"d"-"h"
"e"-"i"
"f"-"i"
"f"-"k"
"i"-"n"
"k"-"n"
"lm"-"I"
"A"-"F"
"B"-"G"
"D"-"H"
"F"-"I"
"F"-"K"
"I"-"N"
"K"-"N"
"LM"-"Q"
}
}} 

\begin{document}

\begin{abstract}
Recently S.A.~Merkulov established a link between differential geometry and homological algebra by giving descriptions of several differential geometric structures in terms of minimal resolutions of props. In particular he described the prop profile of Poisson geometry. In this paper we define a prop such that representations of its minimal resolution in a vector space $V$ are in a one-to-one correspondence with bi-Hamiltonian structures, i.e.~pairs of compatible Poisson structures, on the formal manifold associated to $V$.
\end{abstract}

\maketitle

\section*{Introduction}

Poisson geometry plays a prominent role in Hamiltonian mechanics; the differential equations associated to a Hamiltonian system can be formulated via Poisson structures. The presence of two compatible Poisson structures makes it possible to solve a wide range of integrable Hamiltonian equations, e.g.~the KdV-equations, by providing a hierarchy of integrable vector fields. This kind of geometric structure is called a Poisson pair or a bi-Hamiltonian structure. See e.g.~\cite{Arnold1988} %, \cite{ArnoldGivental}, \cite{DubrovinKricheverNovikov}
for a treatment of Hamiltonian systems, \cite{Weinstein1998} for a survey on Poisson geometry and \cite{Magri2003} for an introduction to bi-Hamiltonian structures.

In the papers \cite{Merkulov2004a}, \cite{Merkulov2005} and \cite{Merkulov2006} S.A.~Merkulov made the %remarkable
discovery that certain differential geometric structures, including Hertling-Manin, Nijenhuis, and Poisson structures, allow descriptions as the degree zero part of minimal resolutions of certain simple algebraic props. Merkulov called such descriptions prop profiles. Apart from the sheer beauty of these observations they provide us with new and surprising links between differential geometry, homological algebra and algebraic topology. For example, the prop profile of Hertling-Manin's weak Frobenius manifolds was shown to be given by a minimal resolution of the operad of Gerstenhaber algebras which in turn is quasi-isomorphic to the chain operad of the little 2-disc operad \cite{Cohen1976}. The prop profile of Poisson geometry, on the other hand, predicts existence of rather mysterious wheeled Poisson structures which can be deformation quantized \cite{Merkulov2008,Merkulov2008a} in a wheeled propic way. Here by wheeled we mean that we allow graphs with oriented cycles which on the geometric side translates to traces of the involved structures. It is an open and interesting question whether or not the associated props have topological meaning as in the case of Hertling-Manin geometry.

The general philosophy of constructing prop profiles can be expressed as follows.
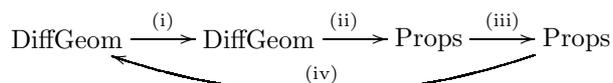
\[
 \xymatrix{
\text{DiffGeom} \ar[r]^{\eqref{A}} & \text{DiffGeom} \ar[r]^{\eqref{B}} & \text{Props}\ar[r]^{\eqref{C}} & \text{Props} \ar@/^20pt/[lll]_{\eqref{D}}
}
\]
\begin{enumerate}
\item \label{A} Extract the fundamental part of the differential geometric structure.
\item \label{B} Translate this fundamental part into a prop $\Pcal$.
\item \label{C} Compute its minimal resolution $\Pcal_\infty$.
\item \label{D} Translate $\Pcal_\infty$ back into a differential geometric structure.
\end{enumerate}

A Poisson structure on a graded manifold $V$ is a graded Lie bracket on the structure sheaf $\Ocal_V$ which acts as a derivation in each argument with respect to the multiplication on $\Ocal_V$. A Poisson structure can equivalently be defined as a bivector field $\G$ of degree two satisfying $[\G,\G]_S=0$. Here the bracket is the Schouten bracket on the polyvector fields on $V$. %and the degree is derived from the grading on $V$.
The fundamental part of this structure translates into the prop $\Poisson$ of \Poissonalgs, i.e.~ of Lie bialgebras with bracket and cobracket differing by one in degree. The prop profile of Poisson geometry, constructed in \cite{Merkulov2006}, is given by its minimal resolution $\Poisson_\infty$. Translating the prop profile back into differential geometry yields a polyvector field $\G$ of degree two, but not necessarily concentrated in $\bvfs$, such that $[\G,\G]_S=0$. We show in this paper (\S \ref{familyofbrackets}) that one can interpret such a polyvector field as a family $\{L_n\}_{n\in\N}$ of $n$-ary brackets on the structure sheaf $\Ocal_V$. These brackets form an $\ELL_\infty$ algebra and act as derivations in each argument with respect to the multiplication in $\Ocal_V$. %If $V$ is concentrated in degree zero we recover a classical Poisson structure.

Our main result is that formal bi-Hamiltonian structures can be derived from a rather simple algebraic structure comprising a Lie bracket of degree one and two compatible Lie cobrackets of degree zero, with the further relations that each cobracket together with the Lie bracket form a \Poissonalg. We call such a structure a \Bipoissonalg\ and denote the corresponding prop by $\Bipoisson$. Using results from \cite{Gan2003}, \cite{Dotsenko2007}, and \cite{Strohmayer2007a}, we show that its dioperadic part is Koszul, which makes it possible to compute its minimal resolution $\Bipoisson_\infty$ (\S \ref{bipoissonresolution}) and leads us to the following conclusion.

\begin{thmalpha}
\label{bipoissoninRn}
There is a one-to-one correspondence between representations of $\Bipoisson_\infty$ in $\R^n$ and formal bi-Hamiltonian structures on $\R^n$ vanishing at the origin.
\end{thmalpha}

In fact we prove a stronger result. When considering representations in arbitrary graded vector spaces we obtain the following result which we prove in Section \ref{geominterpretation}.

\begin{thmalpha}
\label{bipoissoninV}
There is a one-to-one correspondence between representations of $\Bipoisson_\infty$ in a graded vector space $V$ and polyvector fields $\G=\sum_k\leftsub{k}{\G}\hslash^k\in\pvfsh$ on the formal manifold associated to $V$ which depend on the formal parameter $\hslash$ and satisfy the conditions
\begin{enumerate}
\item\label{bipoissonspreadout} $\leftsub{k}{\G}\in\wedge^{\bullet\geq k+1}\Tcal_V$,
\item\label{bipoissondegreetwo} $|\G|=2$,
\item\label{bipoissonschouten} $[\G,\G]_{S}=0$, %$[\G,\G]_{S_\hslash}=0$,
\item\label{bipoissonpointedproperty} $\G|_{0}=0$.
\end{enumerate}
\end{thmalpha}

A pair of Lie algebras are called compatible if the sum of their Lie brackets is again a Lie bracket. We denote the operad of compatible Lie algebras by $\Lietwo$. As a byproduct of the resolution of $\Bipoisson$ we obtain a minimal resolution of $\Lietwo$ (\S \ref{lietworesolution}). We call algebras over this operad $\ELL^2_\infty$ algebras. 

A pair of Poisson structures are called compatible if their brackets are compatible as Lie brackets. We show (\S \ref{abiggerfamily}) how an element $\G\in\pvfsh$ with properties \eqref{bipoissonspreadout}, \eqref{bipoissondegreetwo}, and \eqref{bipoissonschouten} of Theorem \ref{bipoissoninV} corresponds to a family $\{\leftsub{k}{L}_n\}_{n\in\N,1\leq k \leq n}$ of $n$-ary brackets on the structure sheaf $\Ocal_V$. These brackets form an $\ELL^2_\infty$ algebra and act as derivations in each argument with respect to the multiplication in $\Ocal_V$. When $V$ is concentrated in degree zero we obtain precisely a bi-Hamiltonian structure. %pair of compatible Poisson brackets. 
Property \eqref{bipoissonpointedproperty} means that the structure vanish at the distinguished point. By Remark \ref{nonpointedrem} this is not a serious restriction.

To deformation quantize in the propic sense of Merkulov one needs a wheeled propic resolution of $\Bipoisson$. From the dioperadic resolution that we construct one obtains a propic resolution by known results. We note (\S \ref{propstowheeledprops}) that the same obstruction occurs as in the case of $\Poisson$ when trying to extend it to a resolution of wheeled props.

Section \ref{galgebras} comprises definitions of operads, dioperads, properads and props and in Section \ref{galgebraresolutions} we give a formulation of the Koszul duality machinery that enable us to compute resolutions of such algebraic structures. This is done using the unifying approach of $\Gfrak^*$-algebras \cite{Merkulov2007a} which are algebraic structures in which the product is modeled by classes of directed graphs. This makes it possible to define all the above structures as instances of $\Gfrak^*$-algebras differing just in which class of graphs one considers. Also the Koszul duality theory of the respective structures can be expressed as special cases of this unifying theory. Although these sections essentially contain no new material we think their rather encompassing length is motivated for two reasons. Firstly, not all of this material has been expressed in this unifying language and that which has has not been so in this amount of detail. Secondly, to the best of our knowledge their is as of yet no canonical source gathering this material, we wish to keep the paper fairly self-contained and accessible to differential geometers as well as to algebraists.

The rest of the paper is organized as follows. In Section \ref{poissongeometry} we recall basic notions of Poisson geometry and give an interpretation of the prop profile of Poisson structures in terms of a family of brackets. %on the structure sheaf of formal graded manifolds. 
Then, in Sections \ref{propprofileextraction} and \ref{propprofileresolution}, we extract the prop profile of bi-Hamiltonian structures. Finally in Section \ref{geominterpretation} we prove Theorem \ref{bipoissoninRn} and Theorem \ref{bipoissoninV} and interpret the prop profile of bi-Hamiltonian structures as a family of brackets.

A few words about our notation. All vector spaces and tensor products are considered to be over a field $\K$ of characteristic zero unless otherwise specified. For a graded vector space $V=\oplus_{j\in\Z}V_j$ we denote by $V[i]$ the vector space whose graded components are given by $V[i]_j=V_{i+j}$. Given a finite set $S$ we denote its cardinality by $|S|$. By $\N$ we mean the set $\{0,1,2,\dotsc \}$. For $n\in\N$, we denote by $[n]$ the set $\{1,\dotsc,n\}$.  Let $\s_n$ denote the symmetric group of permutations of $[n]$. By $\bbone_n$ we denote the trivial representation of $\s_n$ and by $\sgn_n$ the sign representation. 
\setcounter{tocdepth}{1}
\tableofcontents
\section{$\Gfrak^*$-algebras}
\label{galgebras}

Operads, dioperads and properads are all generalizations of algebras. An algebra consists of a vector space and a product; the generalized structures consist of certain families of these. See \cite{Merkulov2007a} for an introduction to these generalizations via an interpretation of the multiplicative structure of associative algebras in terms of graphs. Below we give thorough definitions of the above structures, as well as their co-versions, using this graph-approach.

\subsection{$\s$-bimodules}

First we define the underlying spaces of our generalized structures.

\begin{defn*}
 An $(\s_m,\s_n)$-\emph{bimodule} is a vector space $M$ with a right action of $\s_n$ and a commuting left action of $\s_m$. A family $\{M(m,n)\}_{m,n\in\N}$ of $(\s_m,\s_n)$-bimodules is called an $\s$-\emph{bimodule}.

A family $\{M(n)\}_{n\in\N}$ of right $(\s_n)$-modules is called an $\s$-\emph{module}.
\end{defn*}

If an $\s$-bimodule $M$ satisfies $M(m,n)=0$ for $m \neq 1$ and all $n$ we can consider it as an $\s$-module since the action of $\s_1$ is trivial. We denote $M(1,n)$ by $M(n)$.

Let $M$ and $N$ be $\s$-bimodules. An \emph{$\s$-bimodule homomorphism} $\theta\colon M\to N$ is a family $\{\theta_{m,n}\colon M(m,n)\to N(m,n) \}_{m,n\in\N}$ of $(\s_m,\s_n)$-bimodule homomorphisms. We will often write $\theta(p)$ for $\theta_{m,n}(p)$.

\subsection{Labeled directed graphs}

Composition of elements of $\s$-bimodules is modeled by graphs. Intuitively we can think of these graphs as 1-dimensional regular CW-complexes with the 1-cells given an orientation. Two subsets of the 1-cells are singled out, directed towards and away from the graph, respectively, and are labeled with integers.

\begin{defn*}

A \emph{labeled directed graph} $G$ is the data
\[
(V_G,E_G,\Phi_G, E^{\inp}_G, E^{\outp}_G, \inp_G ,\outp_G ).
\]
The elements of the set $V_G$ are called the \emph{vertices} of $G$, the elements of the set $E_G$ the \emph{edges}. Further $\Phi_G\colon E_G \to (V_G\times V_G) \sqcup V_G$. The edges in the preimage $\Phi_G^{-1}(V_G)$ are called \emph{external edges} and the edges in the preimage $\Phi_G^{-1}(V_G\times V_G)$ are called \emph{internal}. We denote the internal edges by $E^{\internal}_G$. For an edge $e$ with $\Phi_G(e)=(u,v)$ we say that $e$ is an edge from $u$ to $v$ and in this case we call the vertices $u$ and $v$ \emph{adjacent}.

The set of external edges is partitioned into the sets  $ E^{\inp}_G$ and  $ E^{\outp}_G$ of \emph{global input edges} and \emph{global output edges}, respectively. We denote by $n_G$ and $m_G$ the cardinalities of these sets. The external edges are labeled by integers via the bijections $\inp_G\colon [n_G] \to  E^{\inp}_G$ and $\outp_G\colon  E^{\outp}_G \to [m_G]$.

A \emph{labeled directed $(m,n)$-graph} $G$ is a labeled directed graph with $m_G=m$ and $n_G=n$.

\end{defn*}

Note that the data $(V_G,E^{\internal}_G,\Phi_G|_{E^{\internal}_G})$ is an ordinary directed graph.

There exist a natural right action of $\s_n$ and a commuting left action of $\s_m$ on the class of $(m,n)$-graphs given by permuting the labels. For a labeled directed $(m,n)$-graph $G$ the action of $\tau\in\s_n$ is given by $(\inp_G \tau)(i):=\inp_G\circ\tau(i)$. Similarly $\sigma\in\s_m$ acts to the left by $(\sigma \outp_G)(e):=\sigma\circ\outp_G(e)$, cf.~Figure \ref{graphfig}.

\begin{figure}[h!]
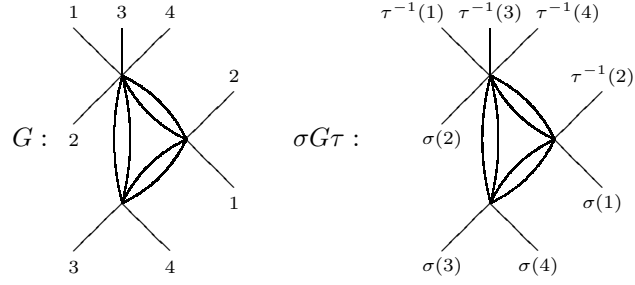

\[
 G:\shs \GraphGnumbered \bhs\bhs \sigma G \tau: \shs \GraphsigmaGtau
\]
\caption{\label{graphfig} The graphs $G$ and $\sigma G \tau$.}
\end{figure}

A \emph{path} from a vertex $u$ to a vertex $v$ in a labeled directed graph is a sequence of edges $e_1,\dotsc , e_r$ such that for some sequence of vertices $u=v_1,\dotsc v_{r+1}=v$ either $\Phi_G(e_i)=(v_i,v_{i+1})$ or $\Phi_G(e_i)=(v_{i+1},v_i)$. A path is called \emph{directed} if $\Phi_G(e_i)=(v_i,v_{i+1})$ for all $i$ and it is called \emph{closed} if $u=v$. A closed directed path is called a \emph{wheel}.
A graph is \emph{connected} if for each pair of vertices there is a path between them.

\subsection{Isomorphisms of graphs}

We are only interested in the structure of the graphs up to a certain level of detail: how many vertices there are, how many internal edges there are in each direction between any two vertices, how many external edges are directed towards and away from each vertex and how they are labeled. Thus we need to define isomorphisms of graphs.

 Let $G$ and $G'$ be labeled directed graphs. An \emph{isomorphism} of labeled directed graphs $\Psi\colon G \to G'$ is a pair
$(\Psi_V,\Psi_E)$, where $\Psi_V\colon V_G\to V_{G'}$  and $\Psi_E\colon E_G\to E_{G'}$ are bijections with the properties
\begin{enumerate}
 \item \label{isoprop-intext} $\Psi_E( E^{\inp}_G)= E^{\inp}_{G'}$ and $\Psi_E( E^{\outp}_G)= E^{\outp}_{G'}$,
 \item \label{isoprop-intstructure} $\Phi_{G'}(\Psi_E(e))=\Psi_V\times\Psi_V(\Phi_{G}(e))$ for all internal edges $e$,
 \item \label{isoprop-extstructure} $\Phi_{G'}(\Psi_E(e))=\Psi_V(\Phi_G(e))$ for all external edges  $e$,
 \item \label{isoprop-globlabel} $\inp_{G'}=\Psi_E\circ\inp_G$ and $\outp_G=\outp_{G'}\circ\Psi_E$.

\end{enumerate}

\begin{ex*} Three graphs, of which the third not is isomorphic to the first two because of the labeling of the edges.
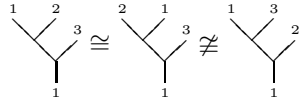
\[ \YYb{1}{2}{3}\iso\YYb{2}{1}{3}\niso\YYb{1}{3}{2}\]
\end{ex*}

\subsection{Classes of graphs}
\label{familiesofgraphs}

From now on we will refer to isomorphism classes of labeled directed graphs simply as graphs. We define the following classes of graphs:
\begin{enumerate}
 \item $\Gfrak^{\circlearrowleft}$ is the class of all graphs. %closed directed paths.
 \label{wheeledpropfamily}
 \item $\Gfrak^{\downarrow}$ is the class of graphs without wheels. %closed directed paths.
 \label{propfamily}
 \item $\Gfrak^{\downarrow}_c$ is the class of connected graphs without wheels. %closed directed paths.
 \label{properadfamily}
 \item $\Gfrak^{\downarrow}_{c,0}$ is the class connected graphs without closed paths (the class of trees).% (i.e.~of genus 0 when considered as CW-complexes).
 \label{dioperadfamily}
 \item $\Gfrak^{\downarrow_1}_c$ is the class of connected graphs without closed paths whose vertices have %at least one edge directed towards it and
exactly one edge directed from it (the class of rooted trees).
 \label{operadfamily}
 \item $\Gfrak^{\downarrow^1_1}_c$ is the class of connected graphs without directed paths whose vertices have exactly one edge directed towards it and exactly one edge directed from it (the class of ladder graphs).
 \label{algebrafamily}
\end{enumerate}
We observe that $\Gfrak^{\downarrow^1_1}_c\subset\Gfrak^{\downarrow_1}_c\subset\Gfrak^{\downarrow}_{c,0}\subset\Gfrak^{\downarrow}_c\subset\Gfrak^{\downarrow}\subset\Gfrak^{\circlearrowleft}$. When depicting graphs of the classes \eqref{propfamily}-\eqref{algebrafamily} we think of them as having a global flow, from global input edges to global output edges, downwards.

From now on let $\Gfrak^{*}$ denote an arbitrary class of the classes  \eqref{wheeledpropfamily}-\eqref{algebrafamily}. We denote by $\Gfrak^*(m,n)$ the subclass of $\Gfrak^*$ consisting of all $(m,n)$-graphs and by $\Gfrak^*_{(i)}$ the subclass consisting of graphs with $i$ vertices.

\subsection{Subgraphs}

In order to describe the associativity of the compositions described by graphs we need to define subgraphs and the notion of contraction of a subgraph in a graph.

Loosely speaking, a subgraph consists of some subset of the vertices of a graph, all edges of the edges in the original graphs attached to the subset of vertices, and an arbitrary global labeling.

 Let $G$ be a graph. A \emph{subgraph} $H$ of $G$ is a graph satisfying
\begin{enumerate}
 \item \label{subgraph-vertexedgeinclusion} $V_H\subset V_G$ and $E_H\subset E_G$,
 \item \label{subgraph-edgepreservation} if $\Phi_G(e)=(u,v)$ and $u,v\in V_H$, then $e\in E_H$ and $\Phi_H(e)=\Phi_G(e)$,
 \item \label{subgraph-internaltoexternal} if $e\in E_G$, $\Phi_G(e)=(u,v)$, $u\notin V_H$ and $v\in V_H$, then $e\in  E^{\inp}_H$ and $\Phi_H(e)=v$, similarly if $u\in V_H$, $v\notin V_H$, then $e\in  E^{\outp}_H$ and $\Phi_H(e)=u$,
 \item \label{subgraph-externaltoexternal} if $e\in  E^{\inp}_G$, $\Phi(e)=v$ and $v\in V_H$, then $e\in E^{\inp}_H$, similarly if $e\in  E^{\outp}_G$
 \end{enumerate}
Note that $\inp_H$ and $\outp_H$ are arbitrary labelings of the global input and output edges of $H$.

% The conditions in the above definition should be interpreted as,
% (\ref{subgraph-vertexedgeinclusion}) the sets of vertices and edges of $H$ are subsets of those of $G$,
% (\ref{subgraph-edgepreservation}) edges from u to v in G are, if u and v are vertices of H, edges from u to v in H
% (\ref{subgraph-internaltoexternal}) internal edges of $G$ are external edges of $H$ if exactly one of the adjacent vertices is in $V_H$,
% global input or output depending which vertex that is not in H
% (\ref{subgraph-externaltoexternal}) external edges of $G$ are, if attached to a vertex of $H$, external edges of $H$.

\subsection{Contraction of subgraphs}

The contraction of a subgraph in a graph can be thought of as replacing all vertices and internal edges of the subgraph with a single vertex.

 Let $H$ be a subgraph of a graph $G$. The \emph{contraction of $H$ in $G$} is the labeled directed graph $G/H$ defined by the same data as $G$ except
\begin{enumerate}
 \item $V_{G/H}=(V_G\setminus V_H)\sqcup\{v_H\}$, where by $v_H$ we denote the vertex into which $H$ is contracted,
 \item $E_{G/H}=E_G \setminus E^{\internal}_H$,
 \item \[\Phi_{G/H}(e)=
\begin{cases} \phantom{(}\Phi_G(e)  &\text{if $e\in E_G\setminus E_H$}\\
(u,v_H)  &\text{if $\Phi_G(e)=(u,v)$ for some $v\in V_H$} \\
(v_H,u)  &\text{if $\Phi_G(e)=(v,u)$ for some $v\in V_H$} \\
\phantom{(}v_H  &\text{if $\Phi_G(e)=v$ for some $v\in V_H$} .\\
\end{cases}\]
\end{enumerate}

We say that a subgraph $H$ of a graph $G\in\Gfrak^*$ is $\Gfrak^*$-\emph{admissible} if both $G/H\in\Gfrak^*$ and $H\in\Gfrak^*$. In Figure \ref{contractionfigb} we see an example of a subgraph $H$ which is $\Gfrak^{\circlearrowleft}$-admissible but not $\Gfrak^{\downarrow}$-admissible.

\begin{figure}[h!]
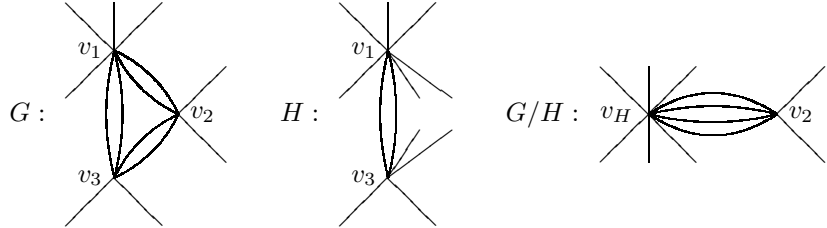

\[
  G:\shs \GraphG \bhs\bhs H:\shs \GraphHb \bhs\bhs G/H:\shs \GraphGHb
\]
\caption{\label{contractionfigb} The contraction $G/H$ of a subgraph $H$ in a graph $G$.}
\end{figure}

\subsection{Decorated graphs}

Compositions of elements of $\s$-bimodules is described by graphs decorated with $\s$-bimodules. When decorating a vertex $v$ with an element $p$ of an $\s$-bimodule $M$ we want to keep track of how we connect $p$ to the internal edges attached to $v$.

We define the set of \emph{local input edges} of $v$ to be
\[
 E^{\inp}_v:=\{e\in E^{\internal}_G \shs |\shs \Phi_G(e)=(u,v) \shs \text{for some $u\in V_G$} \}\cup \{e\in E^{\inp}_G\shs|\shs \Phi_G(e)=v\}
\]
and the set of \emph{local output edges} of $v$ as
\[
 E^{\outp}_v:=\{e\in E^{\internal}_G \shs|\shs \Phi_G(e)=(v,u) \shs \text{for some $u\in V_G$} \}\cup \{e\in E^{\outp}_G\shs|\shs \Phi_G(e)=v\},
\]
Note that we allow edges from a vertex to itself. Such an edge will be both a local input and output edge of the same vertex.

For two finite sets $I$,$J$ of the same cardinality, let $\Bij(I,J)$ denote the set of bijections from $I$ to $J$ and let $\langle I\isoto J\rangle$ denote the vector space generated over $\K$ by $\Bij(I,J)$. If $|I|=n$ there is a natural left action of $\s_{n}$ on $\langle I\isoto [n]\rangle$ given by $\tau g:=\tau\circ g$ for $g\in\Bij(I ,[n])$ and $\tau\in\s_{n}$. Similarly, if $|J|=m$, $\s_{m}$ acts to the right on $\langle [m]\isoto J \rangle$ by $f \sigma:=f \circ\sigma$ for $f\in\Bij([m],J)$ and $\sigma\in\s_{m}$.

We define a vector space by
\[
 M( E^{\outp}_v, E^{\inp}_v):=\langle[m] \isoto  E^{\outp}_v\rangle\otimes_{\s_m} M(m,n)\otimes_{\s_n}\langle E^{\inp}_v \isoto [n]\rangle ,
\]
where $m$ and $n$ are the cardinalities of $ E^{\outp}_v$ and $ E^{\inp}_v$, respectively. Often we will denote an element $f\otimes_{\s_m} p \otimes_{\s_n} g\in M( E^{\outp}_v, E^{\inp}_v)$ by $\overline{p}$ or simply $p$.

\begin{rem*}
 Decorating by $M( E^{\outp}_v, E^{\inp}_v)$ rather than by $M(m,n)$ corresponds to an additional labeling of the internal edges locally, cf.~Figure 1 of \cite[p4869]{Vallette2007b}.
\end{rem*}

We want decorated graphs to extend the notion of tensor products, but for a general graph there is no natural ordering of the vertices.

Let $\{V_i\}_{i\in I}$ be a family of vector spaces indexed by some finite set $I$ with $|I|=k$. The \emph{unordered tensor product} of this family is defined to be
\[
 \bigotimes_{i\in I} V_i:=\left(\bigoplus_{s \in \Bij([k],I)}V_{s(1)}\otimes\dotsb\otimes V_{s(k)}\right)_{\s_k}.
\]
Here we consider the coinvariants with respect to the right action of $\s_k$ on $\Bij([k],I)$. We denote an equivalence class in $\bigotimes_{i\in I} V_i$ by $[v_1\otimes\dotsb\otimes v_k]$, where $v_1\otimes\dotsb\otimes v_k\in V_{s(1)}\otimes\dotsb\otimes V_{s(k)}$ for some $s\in\Bij([k],I)$.

\begin{defn*}
We define the vector space of decorations of a graph $G$ by an $\s$-bimodule $M$ to be $G\langle M \rangle:=\bigotimes_{v\in V_G} M( E^{\outp}_v, E^{\inp}_v)$.

An element of $G\langle M \rangle$ is called a \emph{decorated graph} and is denoted by $(G,[p_1\otimes \dotsb \otimes p_k])$.
\end{defn*}

A \emph{decorated subgraph} of a decorated graph $(G,[p_1\otimes \dotsb \otimes p_k])$ is a decorated graph $(H,[p_{i_1}\otimes \dotsb \otimes p_{i_l}])$ such that $H$ is a subgraph of $G$, $\{i_1,\dotsc, i_l\}=\{i\in [k]\shs | \shs p_i\in M( E^{\outp}_v, E^{\inp}_v) \shs \text{for some $v\in V_H$}\}$, and $i_1 < \dotsb < i_l$.

We define the vector space of decorated $(m,n)$-graphs of $\Gfrak^{*}$ by
\[
 \Gfrak^{*}\langle M \rangle(m,n):= \bigoplus_{G\in\Gfrak^{*}(m,n)}G\langle M \rangle.
\]
There is a natural $(\s_m,\s_n)$-bimodule structure on $\Gfrak^{*}\langle M \rangle(m,n)$ induced by the actions of $\s_m$ and $\s_n$ on $G$, $\sigma( G,[p_1\otimes \dotsb \otimes p_k])\tau:=(\sigma G \tau,[p_1\otimes \dotsb \otimes p_k])$. Thus $\Gfrak^{*}\langle M \rangle(m,n)$ is naturally an $(\s_m,\s_n)$-bimodule. This lets us define the $\s$-bimodule of decorated graphs
\[
 \Gfrak^{*}\langle M \rangle:=\{\Gfrak^{\circlearrowleft}\langle M \rangle(m,n)\}_{m,n\in\N}.
\]

\subsection{$\Gfrak^*$-algebras}
\label{Galgebras}

We are now ready to define the compositions in our generalized structures.

Let $\mu\colon\Gfrak^*\langle M \rangle\to M$ be a homomorphism of $\s$-bimodules. We call such a morphism a \emph{composition product} in $M$. Denote by $\mu_G\colon G\langle M\rangle \to M$ the restriction of $\mu$ to $G\langle M\rangle$. We will write $\mu_G(p_1\otimes \dotsb \otimes p_k)$ for $\mu((G,[p_1\otimes \dotsb \otimes p_k]))$.

Given an $(r,s)$-subgraph $H$ of a graph $G$ we define the morphism
\[
\mu^G_H\colon G\langle M\rangle \to G/H \langle M\rangle
\]
by
\begin{multline*}
 \mu^G_H(G,[\olp_1\otimes\dotsb\otimes \olp_k \otimes \olq_{1}\otimes\dotsb\otimes \olq_{l}]):=\\
(G/H,[( \outp_H^{-1}\otimes_{\s_r}\mu_{H}(\olp_1\otimes\dotsb\otimes \olp_k)\otimes_{\s_s} \inp_H^{-1})\otimes \olq_{1}\otimes\dotsb\otimes \olq_{l}]),
\end{multline*}
where $[\olp_1 \otimes\dotsb\otimes \olp_k]$ is the decoration of $H$.

\begin{defn*}
A \emph{$\Gfrak^*$-algebra} is an $\s$-bimodule $M$ together with a composition product
$\mu:\Gfrak^*\langle M \rangle\to M$ satisfying the associativity condition
 \[
 \mu_G=\mu_{G/H}\circ\mu^G_H
 \]
for each $G\in\Gfrak^*$ and each $\Gfrak^*$-admissible subgraph $H$ of $G$.
\end{defn*}

\subsection{$\Gfrak^*$-coalgebras}
\label{Gcoalgebras}

Let $\Delta\colon M \to \Gfrak^*\langle M \rangle$ be a homomorphism of $\s$-bimodules. We call such a morphism a \emph{cocomposition coproduct} in $M$. Denote by $_G\Delta\colon M \to G\langle M\rangle $ the composition of $\Delta$ with the projection $\Gfrak^*\langle M \rangle \twoheadrightarrow G\langle M\rangle$.

Given an $(r,s)$-subgraph $H$ of a graph $G$ we define the morphism
\[
\mbox{$^{G}_H\hspace{-1pt}\Delta$} \colon G/H\langle M\rangle \to G \langle M\rangle
\]
by
\begin{multline*}
\mbox{$^{G}_H\hspace{-1pt}\Delta$}(G/H,[( \outp_H^{-1}\otimes_{\s_r} \ol{p}_H \otimes_{\s_s} \inp_H^{-1})\otimes \olq_{1}\otimes\dotsb\otimes \olq_{l}]):=\\
(G,[\olp_1\otimes\dotsb\otimes \olp_k \otimes \olq_{1}\otimes\dotsb\otimes \olq_{l}]),
\end{multline*}
where $(H,[\olp_1 \otimes\dotsb\otimes \olp_k])=\Delta_H(\olp_H)$.

\begin{defn*}
 A \emph{$\Gfrak^*$-coalgebra} is an $\s$-bimodule $M$ together with an $\s$-bimodule homomorphisms $\Delta:M\to\Gfrak^*\langle M \rangle$ satisfying the coassociativity condition
 \[
 _G\Delta=\mbox{$^{G}_H\hspace{-1pt}\Delta$} \circ \mbox{$_{G/H}\Delta$}
  \]
for each $G\in\Gfrak^*$ and $\Gfrak^*$-admissible subgraph $H$ of $G$.
\end{defn*}

\subsection{$\Gfrak^\downarrow$-(co)algebras versus $\Gfrak^\circlearrowleft$-(co)algebras}

Some notions related to $\Gfrak^*$-(co)algebras allow simpler expositions when one forgets about $\Gfrak^\circlearrowleft$-(co)algebras. Since we will only implicitly be needing  $\Gfrak^\circlearrowleft$-(co)algebras we avoid the subtleties related to them by restricting our attention to the strict subclasses of $\Gfrak^\circlearrowleft$; from now on let $\Gfrak^*$ be one of the subclasses \eqref{propfamily}-\eqref{algebrafamily} in \S \ref{familiesofgraphs}. See e.g.~\cite{Markl2006a,Merkulov2008} for a treatment of $\Gfrak^\circlearrowleft$-(co)algebras, also called wheeled props (without unit).

\subsection{Graphs decorated by several $\s$-bimodules}

Given a graph $G$ with $|V_G|>1$, we can decorate it by more than one $\s$-bimodule. Let $M_1,\dotsc , M_l$ be $\s$-bimodules and let $V_G=V_1\sqcup\dotsb\sqcup V_l$ be a partition of the set of vertices of $G$. We define the vector space
\[
 G\langle M_1^{V_1},\dotsc M_l^{V_l} \rangle:=\bigotimes_{v\in V_G} M_v,
\]
where $M_v=M_i(E^{\outp}_v,E^{\inp}_v)$ for $v\in V_i$.

Given morphisms of $\s$-bimodules $\theta_1\colon M_1\to N_1,\dotsc , \theta_l\colon M_l\to N_l$ we define the morphism
\[
(\theta_1^{V_1},\dotsc,\theta_l^{V_l})\colon G\langle  M_1^{V_1},\dotsc M_l^{V_l} \rangle \to G\langle  N_1^{V_1},\dotsc N_l^{V_l} \rangle
\]
 by
\[
 (G,[\ol{p}_1\otimes\dotsb\otimes \ol{p}_l])\mapsto  (G,[\ol{\theta_{i_1}(p_1)}\otimes\dotsb\otimes \ol{\theta_{i_k}(p_k)}]),
\]
where $\theta_{i_j}=\theta_r$ when $p_r\in M_r$.

\subsection{Units and counits}
\label{unitsandcounits}

We define the $\s$-bimodule $I$ by
\[
 \begin{cases}
  I(1,1)=\K &   \\
  I(m,n)=0  & \text{for $(m,n)\neq(1,1)$}
 \end{cases}.
\]

Let $G$ be a $(m,n)$-graph satisfying $| E^{\inp}_u|=| E^{\outp}_u|=1$ for all vertices $u\in V_G$ except for one vertex $v$ which then satisfies $| E^{\inp}_v|=n$ and $| E^{\outp}_v|=m$. The maps $\inp_G$ and $\outp_G$ naturally induce maps $\widetilde{\inp}_G\colon [n] \to  E^{\inp}_v$ and $\widetilde{\outp}_G \colon  E^{\outp}_v \to [m]$.

Let $M$ be an $\s$-bimodule. There exists a natural isomorphism 
\[
 G\langle I^{V_G\setminus \{v\}},M^{\{v\}} \rangle \isoto M(m,n)
\]
defined by
\[
 (G,[\ol{c}_1\otimes\dotsb\otimes \ol{c}_{k-1}\otimes \ol{p}])\mapsto (c_1\dotsb c_{k-1})\sigma^{-1} p \tau^{-1},
\]
where $\sigma\in\s_m$ and $\tau\in\s_n$ are permutations such that, for $\ol{p}=f\otimes_{\s_m}p\otimes_{\s_n}g\in M(E^{\outp}_v,E^{\inp}_v)$, $(\tau g) \circ \widetilde{\inp}_G=\Id_{[n]}$ and $\widetilde{\outp}_G \circ (f \sigma)=\Id_{[m]}$.

Let $\mu$ be a composition product in $M$ and let $\eta\colon I \to M$ be an $\s$-bimodule homomorphism. We say that $\eta$ is \emph{unit} with respect to $\mu$ if the following diagram commutes
\[
 \xymatrix{
G\langle I^{V_G\setminus \{v\}},M^{\{ v\}} \rangle \ar[rr]^(.6){(\eta^{V_G\setminus \{v\}},\Id_M^{\{v\}})} \ar@{->}[drr]^(.6)\sim && G\langle M \rangle \ar[d]^{\mu_G} \\
& & M
}.
\]

We denote the element $\eta(1)\in M(1,1)$ by $\bbone$. The above condition is then equivalent to that, for all $G$ as above, the morphism $\mu_G$ satisfies
\begin{equation}
\label{unitdef}
 \mu_G(\bbone\otimes\dotsb\otimes \bbone \otimes (f \otimes_{\s_m} p \otimes_{\s_n} g) \otimes \bbone\otimes\dotsb\otimes \bbone)=\sigma^{-1}p\tau^{-1}.
\end{equation}

On the coside, let $\Delta$ be a cocomposition coproduct in $M$ and let $\epsilon\colon M \to I$ be an $\s$-bimodule homomorphism. We say that $\epsilon$ is \emph{counit} with respect to $\Delta$ if the following diagram commutes
\[
 \xymatrix{
  M\ar[rr]^{_G\Delta} \ar@{->}[drr]^\sim && G\langle M \rangle \ar[d]^{(\epsilon^{V_G\setminus \{v\}},\Id_M^{\{v\}})} \\
& &  G\langle I^{V_G\setminus \{v\}},M^{\{ v\}}\rangle
}.
\]

\subsection{Unital $\Gfrak^*$-algebras and $\Gfrak^*$-coalgebras}

If there exists a morphism $\eta \colon I\to M$, which is a unit with respect to $\mu$, we call the data $(M,\mu,\eta)$ a \emph{unital $\Gfrak^*$-algebra}.

If there exists a morphism $\epsilon \colon M \to I$, which is a counit with respect to $\Delta$, we call the data $(M,\Delta,\epsilon)$ a \emph{counital $\Gfrak^*$-coalgebra}.

\begin{defn*}
We have
\begin{enumerate}
 \item a (co)unital $\Gfrak^{\downarrow}$-(co)algebra is called a \emph{(co)prop},
 \item a (co)unital $\Gfrak^{\downarrow}_c$-(co)algebra is called a \emph{(co)properad},
 \item a (co)unital $\Gfrak^{\downarrow}_{c,0}$-(co)algebra is called a \emph{(co)dioperad},
 \item a (co)unital $\Gfrak^{\downarrow_1}_c$-(co)algebra such that $M(m,n)=0$ if $m\neq 1$ is called an \emph{(co)operad},
 \item a (co)unital $\Gfrak^{\downarrow^1_1}_c$-(co)algebra such that $M(m,n)=0$ if $m,n\neq 1$ is called an \emph{(co)associative (co)algebra}.
\end{enumerate}
\end{defn*}

In \S \ref{galgvsoperad} we show how the above definitions relate to the classical ones.

\subsection{Homomorphisms of $\Gfrak^*$-algebras}
\label{homsofgalgebras}

Let $G$ be graph in $\Gfrak^{*}$ and $v$ be a vertex of $G$. A homomorphism $\theta\colon M \to M'$ of $\s$-bimodules canonically gives rise to a morphism $\theta_v\colon M( E^{\outp}_v, E^{\inp}_v)\to M'( E^{\outp}_v, E^{\inp}_v)$ by
\[
\theta_v(f \otimes_{\s_m} p \otimes_{\s_n} g):=f \otimes_{\s_m} \theta_{m,n}(p) \otimes_{\s_n} g.
\]
We will write $\theta(p)$ for $\theta_v(f\otimes p \otimes g)$. This further extends to a morphism $\theta_G \colon G\langle M \rangle \to G\langle M' \rangle$ by
\[
\theta_G(G,[p_1\otimes \dotsb \otimes p_k]):=(G,\theta(p_1)\otimes\dotsb\otimes\theta(p_k)).
\]
Finally this gives us a morphism of $\s$-bimodules $\theta_{\Gfrak^{*}}\colon\Gfrak^{*}\langle M \rangle \to \Gfrak^{*}\langle M' \rangle$.

Let $(\Pcal,\mu_\Pcal,\eta_\Pcal)$ and $(\Qcal,\mu_\Qcal,\eta_\Qcal)$ be $\Gfrak^*$-algebras. A \emph{$\Gfrak^*$-algebra homomorphism} is a homomorphism of $\s$-bimodules $\theta\colon \Pcal \to \Qcal$ such that $\theta\circ\eta_{\Pcal}=\eta_{\Qcal}$ and for all decorated graphs $G\in\Gfrak^*$ we have $\theta \circ (\mu_\Pcal)_G = (\mu_\Qcal)_G \circ \theta_G$.

\subsection{The endomorphism $\Gfrak^*$-algebra}
We define the endomorphism $\Gfrak^*$-algebra $\End^*_V$ by $\End^*_V(m,n):=\Hom(V^{\otimes n},V^{\otimes m})$. The $(\s_m,\s_n)$ action is given by permuting the input and output. For a graph $G\in\Gfrak^*$, the composition product $\mu_G\colon G\langle \End^*_V \rangle \to \End^*_V$ is defined as the composition of multivariate functions according to $G$. The local labelings of the vertices dictate, in an obvious way, which output is to be plugged into which input of functions decorating adjacent vertices. The global labeling plays a similar role. A unit $\eta\colon I \to \End^*_V$ is given by $\eta(1):=\Id_V$.

\subsection{Representations of $\Gfrak^*$-algebras}

A \emph{representation} of a $\Gfrak^*$-algebra $\Pcal$ in a vector space $V$ is a homomorphism $\rho\colon\Pcal\to\End_V$ of $\Gfrak^*$-algebras. We say that $\rho$ gives $V$ the structure of a \emph{$\Pcal$-algebra}.

We can think of a $\Pcal$-algebra as an assignment of multilinear operations on $V$, possibly with several inputs and outputs, satisfying axioms encoded by the composition product in $\Pcal$.
\section{Resolutions of $\Gfrak^*$-algebras}
\label{galgebraresolutions}

In this section we make definitions of $\Gfrak^*$-algebras presented by generators and relations. To this end we describe the free $\Gfrak^*$-algebra. We also set up the differential graded framework and describe two kinds of resolutions of $\Gfrak^*$-algebras. One kind of resolution is based on an extension of the Koszul duality theory for associative algebras to $\Gfrak^*$-algebras. As the absence of wheels in directed graphs makes a more accessible presentation possible, we restrict our attention in this section to the strict subfamilies of $\Gfrak^\circlearrowleft$, i.e.~in this section $\Gfrak^*$ denotes one of the subfamilies \eqref{propfamily}-\eqref{algebrafamily} defined in \S \ref{familiesofgraphs}.

This section contains no new material; we merely wish to express the results we need from \cite{Ginzburg1994},\cite{Fresse2004},\cite{Gan2003} and \cite{Vallette2007b} in the unifying language of \cite{Merkulov2007a}.

\subsection{Differential graded $\s$-bimodules}

We can also define $\Gfrak^*$-(co)algebras in the differential graded framework.

A \emph{graded $\s$-bimodule} is an $\s$-bimodule $M$ which can be decomposed as $M(m,n):= \bigoplus_{i\in \Z}M(m,n)^i$. We denote by $M^i$ the collection $\{M(m,n)^i\}_{m,n\in\N}$. For an element $p\in M^i$ we write $|p|=i$, and say that $p$ is of degree $i$. We will refer to this degree as the \emph{cohomological degree}.

A \emph{homomorphism $\theta\colon M \to N$ of graded $\s$-bimodules} of degree $j$ is a homomorphism of $\s$-bimodules satisfying $\theta(M^i)\subset N^{i+j}$.

A \emph{differential graded} (dg) \emph{$\s$-bimodule} is a pair $(M,d)$, where $M$ is a graded $\s$-bimodule and $d\colon M \to M$ is a homomorphism of graded $\s$-bimodules of degree one satisfying $d^2=0$.

A \emph{homomorphism $\theta$ of dg $\s$-bimodules} is a homomorphism of graded $\s$-bimodules satisfying $d \circ \theta = \theta \circ d$.

In the differential graded framework we apply the Koszul-Quillen sign rules; whenever a symbol of degree $a$ is moved past a symbol of degree $b$ the sign $(-1)^{ab}$ is introduced.

\subsection{Differential graded  $\Gfrak^*$-algebras and $\Gfrak^*$-coalgebras}
\label{dgalgsandcoalgs}

The differential $d$ of a dg $\s$-bimodule $M$ extends to a differential $d^G$ on the vector space $G\langle M \rangle$ defined by
\[
d^G (G,[p_1\otimes \dotsb \otimes p_k]):=(G,[\sum_{i=1}^{k}(-1)^{|p_1|+\dotsb |p_{i-1}|} p_1\otimes \dotsb \otimes d(p_i) \otimes \dotsb \otimes p_k]).
\]
The grading of $M$ induces a grading on $G\langle M \rangle$ given by $|(G,[p_1\otimes\dotsb\otimes p_k])|=|p_1|+\dotsb+|p_k|$. Together this makes $\Gfrak^*\langle M \rangle$ into a dg $\s$-bimodule.

\begin{defn*}
A \emph{dg $\Gfrak^*$-algebra} is a triple $((\Pcal,d),\mu,\eta)$ where $(\Pcal,\mu,\eta)$ is a $\Gfrak^*$-algebra, $(\Pcal,d)$ is a dg $\s$-bimodule, and $\mu$ is degree zero morphism of dg $\s$-bimodules.
\end{defn*}

Explicitly, the condition that $\mu$ is a morphism of dg $\s$-bimodules is given by $d \mu_G= \mu_G d^G$ for all $G\in\Gfrak^*$. We say that a morphism $d\colon \Pcal\to\Pcal$ is a \emph{$\Gfrak^*$-algebra derivation} if this condition is satisfied.

\begin{defn*}
 A \emph{dg $\Gfrak^*$-coalgebra} is a triple $((\Ccal,d),\Delta,\epsilon)$ where $(\Ccal,\Delta,\epsilon)$ is a $\Gfrak^*$-coalgebra, $(\Ccal,d)$ is a dg $\s$-bimodule, and $\Delta$ is a degree zero morphism of dg $\s$-bimodules.
\end{defn*}

The last condition can be expressed by $_G\Delta d = d^G \mbox{$_G\Delta$}$ for all $G\in\Gfrak^*$. We call an $\s$-bimodule homomorphism $d\colon \Ccal \to \Ccal$ a \emph{$\Gfrak^*$-coalgebra coderivation} if it satisfies this condition.

\subsection{Weight graded $\s$-bimodules and $\Gfrak^*$-(co)algebras}

We will need to consider an extra grading on the objects we study. We call a dg $\s$-bimodule $M$ \emph{weight graded} if it has a decomposition $M=\bigoplus_{s\in\N}M_{(s)}$, where each $M_{(s)}$ is a  dg sub-$\s$-bimodule. This is an extra grading which differs from the cohomological degree in that it does not effect signs, i.e.~the Koszul-Quillen sign rules only apply to homological degree. We call $M_{(s)}$ the weight $s$ part of $M$. A tensor product of weight graded $\s$-bimodules inherits a weight grading by $(M\otimes N)_{(t)}=\bigoplus_{r+s=t}M_{(r)}\otimes N_{(s)}$.

We call a $\Gfrak^*$-algebra $(\Pcal,\mu,\eta)$ \emph{weight-graded} if $\Pcal$ is a weight graded $\s$-bimodule and $\mu$ preserves the weight grading. Note that we necessarily have $\eta(I)\subset \Pcal_{(0)}$.

Similarly we call a $\Gfrak^*$-coalgebra $(\Ccal,\Delta,\epsilon)$ \emph{weight graded} if $\Ccal$ is a weight graded $\s$-bimodule and $\Delta$ preserves the weight grading.

\subsection{Free $\Gfrak^*$-algebras}
\label{freeGalgs}

The free $\Gfrak^*$-algebra $\Free^*(M)$ on an $\s$-bimodule $M$ is characterized by the classical universal property; there is a natural inclusion $\iota\colon M\to\Free^*(M)$ such that given any homomorphism of $\s$-bimodules $\theta\colon M \to \Pcal$ to the underlying $\s$-bimodule of $\Gfrak^*$-algebra, there is a unique extension $\tilde{\theta}$ making the following diagram commute
\[\xymatrix{
 M\ar[dr]^{\iota} \ar[rr]^{\theta}&  & \Pcal \\
 & \Free^*(M)\ar@{-->}[ur]^{\tilde{\theta}} &
}.\]

Here we give an explicit construction. The free non-unital $\Gfrak^*$-algebra, $\Free^*(M)$, on an $\s$-bimodule $M$ has $\Gfrak^*\langle M \rangle$ as underlying $\s$-bimodule. The composition product $\mu\colon\Gfrak^*\langle\Free^*(M)\rangle\to\Free^*(M)$ maps a graph decorated with graphs decorated with $M$ to a graph decorated with $M$. Intuitively we may thing of this composition product as grafting the external edges of the decorating graphs together according to the internal edges of the graph they decorate, leaving the decoration by $M$ unchanged, except for a minor modification of the internal labeling.

To be more precise, for a graph $G\in\Gfrak^*$, the morphism $\mu_G$ maps
\[
(G,[\ol{(G_1,[\olp^1_1\otimes\dotsb\otimes \olp^1_{k_l}])}\otimes\dotsb\otimes \ol{(G_k,[\olp^k_1\otimes\dotsb\otimes \olp^k_{l_k}])}] ) \in G\langle(\Gfrak^*\langle M \rangle ) \rangle
\]
to
\[
(G(G_1,\dotsc,G_k),[\wtp^1_1\otimes\dotsb\otimes \wtp^k_{l_k}])\in G(G_1,\dotsc , G_k)\langle M \rangle,
\]
where $G(G_1,\dotsc,G_k)$ is the result of the grafting and $\wtp^a_b$ is equal to $\olp^a_b$ up to a modification of the labeling to keep track of how the $p^a_b$ connect to the grafted graph. We describe in detail the graph $G(G_1,\dotsc G_k)$ as well as the modification of the labeling in \S \ref{freegrafting} of the appendix.

Since it does not matter in which order we graft the edges, the associativity condition $\mu_G=\mu_{G/H}\circ\mu^G_H$ is immediate.

To define a unit of $\Free^*(M)$ we have to add a special graph, $|$, to $\Gfrak^*$ consisting of a single edge and no vertices. The space of decorations is defined as $|\langle M \rangle := \K$, in analogy with the tensor product of zero factors. We define the grafting $G(|,\dotsc,| , G', |,\dotsc,|):=\sigma^{-1} G' \tau^{-1}$, where $\sigma$ and $\tau$ are defined as in~\eqref{unitdef} in \S \ref{unitsandcounits}. The unit is then defined by $\bbone:=(|,[1])$.

The inclusion $\iota \colon M \to \Free^*(M)$ is defined as follows, for an element $p\in M(m,n)$, its image $\iota(p)$ is the decorated one vertex $(m,n)$-graph $(G,[f\otimes_\s p \otimes_\s g])$ such that $g \circ \inp_G=\Id_{[n]}$ and $\outp_G \circ f=\Id_{[m]}$.

We will usually omit the $*$ and denote a free $\Gfrak^*$-algebra simply by $\Free(M)$ when it is clear which family of graphs we consider.

\subsection{Cofree $\Gfrak^*$-coalgebras}

The cofree $\Gfrak^*$-coalgebra on an $\s$-bimodule $M$ is characterized by the universal property obtained by reversing all arrows in the diagram characterizing free $\Gfrak^*$-algebras. Its underlying $\s$-bimodule is also $\Gfrak^*\langle M \rangle$. The cocomposition product $\Delta$ is defined as follows. Let $_G\Delta$ be the projection of $\Delta$ onto $G\langle M \rangle$, then for a decorated graph $X=(\widetilde{G},[p_1\otimes\dotsb\otimes p_k])$ the image of $X$ under $_G\Delta$ is the sum over all decorated graphs $Y=(G,[(G_1,[p^1_1\otimes\dotsb\otimes p^1_{k_l}])\otimes\dotsb\otimes (G_k,[p^k_1\otimes\dotsb\otimes p^k_{l_k}])])$ such that $\mu(Y)=X$ in the free $\Gfrak^*$-algebra on $M$. The counit is given by $\epsilon\colon (|,[1])\to 1$ and zero otherwise.

\subsection{Quadratic $\Gfrak^*$-algebras}

As for associative algebras we want to give presentations of $\Gfrak^*$-algebras in terms of generators and relations.

An \emph{ideal} of a $\Gfrak^*$-algebra $\Pcal$ is a sub-$\s$-bimodule $\Jcal$ satisfying $\mu_G(p_1\otimes \dotsb \otimes p_k)\in\Jcal$ whenever at least one of the $p_i$ is in $\Jcal$. We denote the ideal generated by a subset $J \subset\Pcal$ by $(J)$.
%Given a sub-$\s$-bimodule $\Jcal$ the ideal generated by $\Jcal$ is

Let $\Pcal$ be a $\Gfrak^*$-algebra and $\Jcal$ be an ideal of $\Pcal$. The \emph{quotient $\Gfrak^*$-algebra} $\Pcal/\Jcal$ is defined by  $\Pcal/\Jcal(m,n):=\Pcal(m,n)/\Jcal(m,n)$. If $\Pcal$ is weight graded and the ideal $\Jcal$ is homogeneous with respect to this weight grading, i.e.~$\Jcal=\bigoplus_{s \in \N}\Jcal_{(s)}$ and $\Jcal_{(s)}=\Jcal\cap\Pcal_{(s)}$, then the quotient $\Pcal/\Jcal$ inherits a weight grading from $\Pcal$.

The free $\Gfrak^*$-algebra has a natural weight grading by the number of vertices of a decorated graph, $\Free^*(M)=\bigoplus_{s \in \N}\Free^*_{(s)}(M)$, where $\Free^*_{(s)}(M):=\Gfrak^*_{(s)}\langle M \rangle$. 
%This decomposition is compatible with the composition product $\mu$, i.e.~$\mu_G(p_1\otimes\dotsb\otimes p_k)\in\Free^*_{(i_1+\dotsb + i_k)}(M)$ if $p_j\in\Free^*_{(i_j)}(M)$.

\begin{defn*}
 A \emph{quadratic $\Gfrak^*$-algebra} is a $\Gfrak^*$-algebra $\Pcal=\Free^*(M)/(R)$, where $R \subset\Free^*_{(2)}(M)$.
\end{defn*}

\begin{ex}[Dioperad of Lie bialgebras]
\label{liebi}
Let $M$ be the $\s$-bimodule given by $M(1,2)=\bbone_1\otimes\sgn_2$, $M(2,1)=\sgn_2\otimes\bbone_1$, and $M(m,n)=0$ for other $m,n$. We denote a graph decorated with the natural basis element of $M(1,2)$ by $\Ysmall$ and a graph decorated with the basis element of $M(2,1)$ by $\coYsmall$. Consider the quadratic dioperad $\Liebi=\QO{M}{R}$ where $R=R(1,3)\sqcup R(3,1) \sqcup R(2,2)$, with $R(i,j)\subset\Free(M)(i,j)$, is the following set of relations
\begin{align}
R(1,3): &\YY{1}{2}{3}+\YY{2}{3}{1}+\YY{3}{1}{2}  \bhs\bhs R(3,1): \coYcoY{1}{2}{3}+\coYcoY{2}{3}{1}+\coYcoY{3}{1}{2} \label{LBjacobicojacobi}\\
R(2,2): &\YcoY{1}{2}{1}{2}-\coYopY{1}{2}{1}{2}+\coYopY{2}{1}{1}{2}+\coYopY{1}{2}{2}{1}-\coYopY{2}{1}{2}{1}. \label{LBcompatibility}
\end{align}
A representation $\rho\colon\Liebi\to\End_V$ in a vector space $V$ makes $V$ into a Lie bialgebra. The Lie bracket is given by $\rho(\Ysmall)\colon V^{\otimes 2}\to V$ and the Lie cobracket by $\rho(\coYsmall)\colon V\to V^{\otimes 2}$. That $\rho$ is map of dioperads ensures that the Jacobi and co-Jacobi identities \eqref{LBjacobicojacobi} are satisfied as well as the compatibility of the brackets \eqref{LBcompatibility}.
\end{ex}
See e.g.~\cite{Ginzburg1994} for a treatment of quadratic operads, \cite{Gan2003} for quadratic dioperads and \cite{Vallette2007b} for quadratic properads and props.

\subsection{Connected $\Gfrak^*$-(co)algebras}

We call an $\s$-bimodule \emph{connected} if $M(m,0)=0$ for all $m$, $M(0,n)=0$ for all $n$, and $M(1,1)=\K$.

A weight graded $\s$-bimodule $M$ is called \emph{connected} if $M$ is connected as a dg $\s$-bimodule, $M_{(0)}(1,1)=\K$, and $M_{(0)}(m,n)=0$ for other $m,n$. We also require that if $r<0$, then $M_{(r)}(m,n)=0$ for all $m,n$.

We call a (weight graded) $\Gfrak^*$-(co)algebra \emph{connected} if the underlying $\s$-bimodule is connected.

\subsection{(Co)augmented $\Gfrak^*$-(co)algebras}

We can give the $\s$-bimodule $I$, with
\[
 \begin{cases}
  I(1,1)=\K &   \\
  I(m,n)=0  & \text{for $(m,n)\neq(1,1)$},
 \end{cases}
\]
a $\Gfrak^*$-algebra structure by defining the composition product $\mu_G(c_1\otimes\dotsb\otimes c_k):=c_1 \dotsb c_k$, the product of the scalars, the unit $\eta$ being the identity $I\to I$.

An \emph{augmentation} of a $\Gfrak^*$-algebra $\Pcal$ is a morphism of $\Gfrak^*$-algebras $\epsilon\colon \Pcal \to I$. We define the \emph{augmentation ideal} of $\Pcal$ by $\bar{\Pcal}(m,n):=\ker(\epsilon_{m,n})$.

We can also give the $\s$-bimodule $I$ a $\Gfrak^*$-coalgebra structure by $_G\Delta(c):=(G,[c \cdot1\otimes\dotsb\otimes 1])$, the counit $\epsilon$ being the identity $I\to I$.

A \emph{coideal} of a $\Gfrak^*$-coalgebra $\Ccal$ is a sub-$\s$-bimodule $\Jcal$ such that
\[
 \Delta_G(\Jcal)\subset\bigoplus_{v\in V_G}G\langle \Ccal^{V_G\setminus\{v\}},\Jcal^{\{v\}} \rangle.
\]

A \emph{coaugmentation} of a $\Gfrak^*$-coalgebra $\Ccal$ is a morphism of $\Gfrak^*$-coalgebras $\eta\colon I \to \Ccal$. We define the \emph{coaugmentation coideal} of $\Ccal$ by $\bar{\Ccal}(m,n):=\coker(\eta_{m,n})$.

The augmentation ideal of the free $\Gfrak^*$-algebra and the coaugmentation coideal of the cofree $\Gfrak^*$-coalgebra are given by the same $\s$-bimodule $\bar{\Free}^*\langle M \rangle=\bar{\Free}^{*,c}\langle M \rangle=\bigoplus_{s\geq 1}\Gfrak^*_{(s)}\langle M \rangle$.

\subsection{Suspension and desuspension}

 The \emph{suspension} $\Sigma M$ of a dg $\s$-bimodule $M$ is defined as $(\Sigma M)(m,n):=\K s\otimes M(m,n)$, where $s$ is an element of degree $1$. We define the \emph{desuspension} $\Sigma^{-1} M$ by $(\Sigma^{-1} M)(m,n):=\K s^{-1}\otimes M(m,n)$, where  $s^{-1}$ is an element of degree $-1$. Thus $(\Sigma M)_i=M_{i-1}$ and $(\Sigma^{-1} M)_i=M_{i+1}$.

\subsection{Derivations of free $\Gfrak^*$-algebras}
\label{derivationsoffree}

Let $\Free^*(M)$, be the free $\Gfrak^*$-algebra on an $\s$-bimodule $M$ and let $\theta\colon M \to \Free^*(M)$ be an $\s$-bimodule homomorphism. Such a morphism $\theta$ determines a $\Gfrak^*$-algebra derivation $_\theta d\colon\Free^*(M)\to\Free^*(M)$. The morphism $\theta$ is itself determined by morphisms $_G\theta\colon M \to G\langle M \rangle$, with $|V_G|\geq 1$. For a pair of graphs $H\subset G$ we define the morphism $\mbox{$^{\phantom{,}G}_H\hspace{-1pt}\theta$}\colon G/H\langle M \rangle \to G\langle M \rangle$ as in \S \ref{Gcoalgebras}, then $_\theta d$ defined by
\[
 _\theta d|_{\widetilde{G}\langle M \rangle} := \sum_{G/H=\widetilde{G}}\mbox{$^{\phantom{,}G}_H\hspace{-1pt}\theta$}
\]
can readily be checked to satisfy the derivation property. The above sum is over all pairs $H$, $G$ such that $H$ is an admissible subgraph of $G$ up to the global labeling of $H$ since \mbox{$^{\phantom{,}G}_H\hspace{-1pt}\theta$} is not dependent on this labeling. Since $_G\theta$ applied to a fixed element of $M$ is non-zero for only finitely many $G$ so is true also for \mbox{$^{\phantom{,}G}_H\hspace{-1pt}\theta$}.

Conversely a derivation $d$ of the free $\Gfrak^*$-algebra $\Free^*(M)$ is determined by its restriction $d|_M \colon M \to \Free^*(M)$. Indeed,
\begin{multline*}
 d(G,[p_1\otimes\dotsb\otimes p_k])=\\
d \mu_G((G_1,[p_1])\otimes\dotsb\otimes(G_k,[p_k])) = \mu_G d^G((G_1,[p_1])\otimes\dotsb\otimes(G_k,[p_k]))=\\
\sum_{i=1}^k(-1)^{(|p_1|+\dotsb+|p_{i-1}|)|d|} \mu_G(  (G_1,[p_1])\otimes\dotsb\otimes (G_i,[d p_i]) \otimes\dotsb\otimes(G_k,[p_k]))=\\
\sum_{i=1}^k(-1)^{(|p_1|+\dotsb+|p_{i-1}|)|d|} (G,[p_1\otimes\dotsb\otimes d p_i \otimes\dotsb\otimes p_k ]).
\end{multline*}
Here the one-vertex graphs $G_i$ and the local labelings of the $p_i$ are appropriately chosen so as to satisfy the above equalities as well as $(G_i,[p_i])\isomapsto p_i$ under the isomorphism defined in \S \ref{unitsandcounits}.

Combining the last two observations we conclude the following. 

\begin{prop}
There is a one-to-one correspondence between $\Gfrak^*$-algebra derivations of $\Free^*(M)$ and $\sbimodule$ homomorphisms  $M \to \Free^*(M)$.
\end{prop}

\subsection{Coderivations of free $\Gfrak^*$-coalgebras}
\label{coderivationsofcofree}

%Manual hyphenation added
Let $\Free^{*,c}(M)$ be the free $\Gfrak^*$-coalge-bra on an $\s$-bimodule $M$ and let $\theta\colon \Free^{*,c}(M) \to M$ be an $\s$-bimodule homomorphism. %satisfying $\theta(\Free^{*,c}_{(0)}(M))=0$. 
Such a $\theta$ determines a $\Gfrak^*$-coalgebra coderivation $d_\theta\colon\Free^{*,c}(M)\to\Free^{*,c}(M)$ as follows. The morphism $\theta$ is itself determined by morphisms $\theta_G\colon G\langle M \rangle \to M$. For a pair of graphs $H\subset G$, we define the morphism $\theta^G_H\colon G\langle M \rangle \to G/H\langle M \rangle$ as in \S \ref{Galgebras}, then $d_\theta$ defined by
\[
 d_\theta|_{G\langle M \rangle} := \sum_{H \subset G}\theta^G_H,
\]
can readily be checked to satisfy the coderivation property. Here the sum is over, up to the labeling of $H$, all $\Gfrak^*$-admissible subgraphs $H$ of $G$.

Conversely a coderivation $d$ of the cofree $\Gfrak^*$-coalgebra $\Free^{*,c}(M)$ is uniquely determined by the projection $\pi_M d \colon \Free^{*,c}(M)\to M$. First we observe that
\[
d( \Free^{*,c}_{(s)}(M))\subset \bigoplus_{r\leq s}\Free^{*,c}_{(r)}(M).
\]
This claim is verified by induction on the number of vertices. Now suppose $d|_{ \Free^{*,c}_{(r)}(M)}$ is known for all $r< s$ and consider $X=(G,[p_1\otimes\dotsb\otimes p_s])$. First we note that $d(X)$ is a sum of decorated graphs with at most $s$ vertices. Next for $G'\in\Gfrak^*_{(2)}$ we observe that $_{G'}\Delta d(X)$, if non-zero, consists of terms where either one of the vertices is decorated with $(|,[1])$ or both vertices are decorated with graphs with at most $s-1$ vertices. It is clear that in order to determine the part of $d(X)$ which consists of graphs with more than one vertex, it is enough to know the part of $_{G'}\Delta d(X)$ without trivially decorated vertices, for all $G'\in\Gfrak^*_{(2)}$. Thus if we consider only such terms in the equality $_{G'}\Delta d(X)=d^{G'}\mbox{$_{G'}\Delta$}(X)$, then in the right hand side $d$ is applied only to graphs with less than $s$ vertices and is therefore known by the induction assumption. Hence $d(X)$ is fully known if we only know the projection of $d$ to $\Free^{*,c}_{(1)}(M)\iso M$. We have proved the following.

\begin{prop}
There is a one-to-one correspondence between $\Gfrak^*$-coalgebra coderivations of $\Free^{*,c}(M)$ and $\sbimodule$ homomorphisms  $\Free^*(M)\to M$.
\end{prop}

\subsection{Quasi-free dg $\Gfrak^*$-algebras and $\Gfrak^*$-coalgebras}

A free $\Gfrak^*$-algebra on a dg $\s$-bimodule $(M,d)$ has a natural differential induced by $d$, as defined in \S \ref{dgalgsandcoalgs}. We will consider also free $\Gfrak^*$-algebras where the differential differs from the differential freely generated by $d$. We call a free $\Gfrak^*$-algebra $\Free^*(M)$ with a differential $_\theta\delta=d+\mbox{$_\theta d$}$, where $_\theta d$ is a derivation determined by a morphism $\theta\colon M \to \Free^*(M)$ (cf.~\S \ref{derivationsoffree}), a \emph{quasi-free} $\Gfrak^*$-algebra.

Similarly we call a cofree $\Gfrak^*$-coalgebra $\Free^{*,c}(M)$ \emph{quasi-cofree} if its codifferential is a sum $\delta_\theta=d+d_\theta$ of the codifferential induced by the one on $M$ and a coderivation $d_\theta$ determined by a morphism $\theta\colon \Free^*(M)\to M$ (cf.~\S \ref{coderivationsofcofree}).

\subsection{Quasi-free resolutions}

A \emph{quasi-free resolution} of dg $\Gfrak^*$-algebra $(\Pcal,\delta)$ is a quasi-free $\Gfrak^*$-algebra $(\Free^*(M),d+\mbox{$_\theta d$})$ together with a quasi-isomorphism $\phi\colon\Free^*(M)\\\to\Pcal$. If \mbox{$_\theta d$} satisfies $\mbox{$_\theta d$}(M)\subset\bigoplus_{i\geq 2}\Free^*_{(i)}(M)$ we call the resolution \emph{minimal}.

\subsection{Bar and cobar constructions}

Let $\Pcal$ be a dg $\Gfrak^*$-algebra. Consider the cofree $\Gfrak^*$-coalgebra $\Free^{*,c}(\Sigma^{-1}\bar{\Pcal})$. It comes equipped with the codifferential $d$ induced by the differential of $\Pcal$, cf.~\S \ref{dgalgsandcoalgs}. The restriction of $\mu_\Pcal$ to $\Gfrak^*_{(2)}\langle \Pcal \rangle$ induces a degree one morphism $\theta\colon\Gfrak^*_{(2)}\langle \Sigma^{-1}\bar{\Pcal} \rangle\to \Sigma^{-1}\bar{\Pcal}$. By \S \ref{coderivationsofcofree}, the morphism $\theta$ determines a coderivation $_\theta d$ of $\Free^{*,c}(\Sigma^{-1}\bar{\Pcal})$. The associativity of $\mu_\Pcal$ implies that $_\theta d^2=0$. That $\Pcal$ is a dg $\Gfrak^*$-algebra implies that $d (\mbox{$_\theta d$}) + (\mbox{$_\theta d$}) d=0$. Thus we see that $\delta:=d+\mbox{$_\theta d$}$ satisfies $\delta^2=0$. We define the \emph{bar construction} of $\Pcal$ to be the quasi-cofree $\Gfrak^*$-coalgebra $\BC^*(\Pcal):=(\Free^{*,c}(\Sigma^{-1}\bar{\Pcal}),\delta)$.

Now let $\Ccal$ be a dg $\Gfrak^*$-coalgebra. We define the \emph{cobar construction} of $\Ccal$ to be the quasi-free $\Gfrak^*$-algebra $\CBC^*(\Ccal):=(\Free^{*}(\Sigma\bar{\Ccal}),\delta)$ where the differential $\delta:=d+d_\theta$ is defined as follows. The $\Gfrak^*$-algebra $\Free^{*}(\Sigma\bar{\Ccal})$ has a differential $d$ induced by the codifferential of $\Ccal$, cf.~\S \ref{dgalgsandcoalgs}. The projection of $\Delta_\Ccal$ to $\Gfrak^*_{(2)}\langle \Ccal \rangle$ induces a degree one morphism $\theta\colon\Sigma\bar{\Ccal}\to\Gfrak^*_{(2)}\langle \Sigma\bar{\Ccal} \rangle$. By \S \ref{derivationsoffree}, $\theta$ determines a derivation $d_\theta$ of $\Free^{*}(\Sigma\bar{\Ccal})$. The coassociativity of $\Delta$ implies that $d_\theta ^2=0$.  That $\Delta$ is a morphism of dg $\s$-bimodules implies that $d (d_\theta) + (d_\theta) d=0$. By the above observations we see that $\delta^2=0$.

When we do not want to emphasize which family of graphs we are considering we will usually omit the $*$ from the notation of the bar and cobar constructions.

\subsection{The bar-cobar resolution}
Applying first the bar and then the cobar construction to a $\Gfrak^*$-algebra $\Pcal$ yields a quasi-free resolution of $\Pcal$.

\begin{thm*}Let $\Pcal$ be a connected dg $\Gfrak^*$-algebra, where $\Gfrak^*$ is one of $\Gfrak^{\downarrow^1_1}_c$, $\Gfrak^{\downarrow_1}_c$, $\Gfrak^{\downarrow}_{c,0}$, and $\Gfrak^\downarrow_c$. In this case the morphism
\[
\Free^*(\Sigma\bar{\Free}^{*,c}(\Sigma^{-1} \bar{\Pcal}))\to\Pcal
\]
induced by the projection
\[
\Sigma\bar{\Free}^{*,c}(\Sigma^{-1} \bar{\Pcal})\to \Sigma\bar{\Free}^{*,c}_{(1)}(\Sigma^{-1} \bar{\Pcal})\iso\bar{\Pcal}\subset\Pcal
\]
induces a quasi-isomorphism of dg $\Gfrak^*$-algebras
\[
\CBC(\BC(\Pcal))\isoto \Pcal .
\]
\end{thm*}

This was proved %for associative algebras in \cite{?},
for operads in \cite{Ginzburg1994}, for dioperads in \cite{Gan2003}, and for properads in \cite{Vallette2007b}. The problem with the bar-cobar resolution is that it can be very difficult to compute explicitly. Fortunately there is a large class of $\Gfrak^*$-algebras for which there exists a more easily computable resolution.

\subsection{Koszul $\Gfrak^*$-algebras}

In addition to the weight grading given by the number of vertices, the cofree $\Gfrak^*$-coalgebra on a weight graded $\s$-bimodule $M$ inherits another weight grading, the total weight,
\[
 \Free^{*,c}(M)_{(s)}:=\bigoplus_{\substack{G\in\Gfrak^* \\ \{v_1\}\cup\dotsb\cup\{v_k\}=V_G\\ s_1+\dotsb + s_k=s}}G\langle (M_{(s_1)})^{v_1},\dotsc, (M_{(s_k)})^{v_k} \rangle.
\]
For a weight graded $\s$-bimodule $M$ concentrated in positive weight we observe that
\[
\begin{cases}
\Free_{(s)}^{*,c}(M)_{(s)}=\Free_{(s)}^{*,c}(M_{(1)})\\
\Free_{(r)}^{*,c}(M)_{(s)}=0 & \text{for $r>s$.}
\end{cases}
\]

Now consider the bar construction $\BC(\Pcal)$ on a connected weight graded $\Gfrak^*$-algebra $\Pcal$. By the above observations we see that $\BC(\Pcal)$ is bi-graded by the number of vertices and the total weight. We also observe that $\Sigma^{-1}\bar{\Pcal}$ is concentrated in positive weight since $\Pcal$ is connected. By construction we see that $\mbox{$_\theta d$}(\BC_{(r)}(\Pcal)_{(s)})\subset\BC_{(r-1)}(\Pcal)_{(s)}$. The compatibility of $_\theta d$ and $d$ yields a complex of dg $\s$-bimodules
\[
 0 \to \BC_{(s)}(\Pcal)_{(s)} \to \BC_{(s-1)}(\Pcal)_{(s)}  \to ...
\]
One can show that the weight graded sub-$\s$-bimodule given by
\[
(\Pcal^{\antishriek})_{(s)}:=\Ho_s(\BC^*_{(\bullet)}(\Pcal)_{(s)},\mbox{$_\theta d$})=\ker( \mbox{$_\theta d$}\colon\BC_{(s)}(\Pcal)_{(s)} \to \BC_{(s-1)}(\Pcal)_{(s)})
\]
is a weight graded sub-$\Gfrak^*$-coalgebra of $\BC^*(\Pcal)$. We call $\Pcal^{\antishriek}$ %this $\Gfrak^*$-coalgebra
the \emph{Koszul dual} of $\Pcal$ and we say that $\Pcal$ is \emph{Koszul} if the inclusion $\Pcal^{\antishriek}\hookrightarrow \BC^*(\Pcal)$ is a quasi-isomorphism. It is shown in the above mentioned references that a Koszul $\Gfrak^*$-algebra is necessarily quadratic.% ref[?].
\begin{rem*}
 Note that the Koszul dual is defined as the \emph{homology} of $(\BC^*(\Pcal),\leftsub{\theta}{d})$ with respect to the weight grading. The codifferential raises the cohomological degree by one but lowers the weight by one. 
\end{rem*}

\subsection{The Koszul resolution}
\label{koszulresolution}

For Koszul $\Gfrak^*$-algebras we have the following result.
\begin{thm*}Let $\Pcal$ be a Koszul dg $\Gfrak^*$-algebra, where $\Gfrak^*$ is one of $\Gfrak^{\downarrow_1}_c$, $\Gfrak^{\downarrow}_{c,0}$, and $\Gfrak^\downarrow_c$. In this case the morphism of the bar-cobar resolution induces a quasi-isomorphism
\[
\CBC(\Pcal^{\antishriek})\isoto \Pcal .
\]
\end{thm*}
For a Koszul $\Gfrak^*$-algebra $\Pcal$ we denote this resolution by $\Pcal_\infty$. Representations of $\Pcal_\infty$ yield strongly homotopy, also called infinity, versions of the algebras corresponding to $\Pcal$; e.g.~algebras over the operad $\Lie_\infty$ are called strongly homotopy Lie algebras or $\ELL_\infty$ algebras.

If $\Pcal$ is a Koszul $\Gfrak^*$-algebra with zero differential, then all we need to know in order to compute the differential of $\CBC(\Pcal^{\antishriek})$ is the structure of the cocomposition coproduct of $\Pcal^{\antishriek}$. Next we will consider a shortcut to determining this cocomposition coproduct.

\subsection{Koszul dual $\Gfrak^*$-algebras}
\label{koszuldualGalgebra}

To a quadratic $\Gfrak^*$-algebra there is an associated dual $\Gfrak^*$-algebra defined as follows.

Let $M$ be an $\s$-bimodule. The \emph{Czech dual} $\s$-bimodule $M^{\cz}$ of $M$ is defined by $M^{\cz}(m,n):=\sgn_m \otimes M(m,n)^*\otimes\sgn_n$. Now consider the free $\Gfrak$-algebra on a connected $\s$-bimodule $M$ satisfying in addition that $M(1,1)=0$ and that $M(m,n)$ is finite dimensional for all $(m,n)$. The components $\Free^*_{(s)}(M)(m,n)$ are then all finite dimensional and the linear dual $(\Free^*(M))^*$ is naturally isomorphic to $\Free^{*,c}(M^*)$ as $\Gfrak^*$-coalgebras. This isomorphism induces a pairing
\[
\langle\_ ,\_\rangle \colon \Free^*_{(2)}(M)\otimes\Free^*_{(2)}(M)\to\K
\]
defined by
\[
(G,[e^*_{a}\otimes e^*_{b}])\otimes (G',[e_{c}\otimes e_{d}])\mapsto \delta_{G,G'}\delta_{a,c}\delta_{b,d},
\]
where the $\delta$:s are Kronecker deltas, $\{e_i\}$ is a basis of $M$, $\{e^*_i\}$ the dual basis, and we assume in the case $G=G'$ that $e_a$ decorates the same vertex as $e_c$.

Now let $\Pcal=\QO{M}{R}$ be a quadratic $\Gfrak^*$-algebra such that $M$ satisfies the above conditions. Let $R^\perp$ be a subset of $\Free^*_{(2)}(M^{\cz})$ satisfying that $(R^\perp)_{(2)}$ is orthogonal to $(R)$ with respect to the pairing $\langle\_ ,\_\rangle$. We define the \emph{Koszul dual $\Gfrak^*$-algebra} of $\Pcal$ to be $\Pcal^!:=\QO{M^{\cz}}{R^\perp}$.

The Koszul dual $\Gfrak^*$-algebra $\Pcal^!$ of a quadratic $\Gfrak^*$-algebra $\Pcal$ relates to the Koszul dual $\Pcal^{\antishriek}$ in the following way
\[
(\Pcal^{\antishriek})_{(s)}(m,n)\iso\Sigma^{-s}((\Pcal^!)_{(s)}(m,n))^{\cz},
\]
where the isomorphism is of  $\Gfrak^*$-coalgebras. Thus, computing the Koszul dual $\Gfrak^*$-algebra and its composition product gives us an accessible way of determining the differential of the bar construction on $\Pcal^{\antishriek}$.
\section{Poisson geometry}
\label{poissongeometry}

In this section we recall basic facts concerning Poisson structures. We also give an interpretation of extended Poisson structures on formal graded manifolds as a family of brackets comprising an $\ELL_\infty$ algebra on the structure sheaf of the manifold.

\subsection{Classical Poisson geometry}
\label{poissongeompar}

Let $\Mcal$ be a manifold and denote by $\Ocal_\Mcal$ the \emph{structure sheaf} of $\Mcal$, i.e.~the sheaf of commutative $\K$-algebras of smooth functions on $\Mcal$. A \emph{Poisson bracket} on $\Mcal$ is an operation $\{\_,\_\}\colon \Ocal_{\Mcal} \otimes \Ocal_{\Mcal}\to \Ocal_{\Mcal}$ satisfying the following  properties
\begin{enumerate}
 \item
\label{skewsymmetry}
 $\{f,g\}=-\{g,f\}$ \bhs (skew-symmetry)
 \item
\label{jacobiproperty}
  $\{f,\{g,h\}\}+\{g,\{h,f\}\}+\{h,\{f,g\}\}=0$ \bhs (Jacobi identity)
 \item
\label{derproperty}
 $\{f,gh\}=\{f,g\}h+g\{f,h\}$ \bhs (Leibniz property of $\{f,\_\}$).
\end{enumerate}
Thus a Poisson bracket is a Lie bracket on $\Ocal_\Mcal$ which in each argument acts as a derivation with respect to the multiplication of smooth functions on $\Ocal_\Mcal$. We call the pair $(\Mcal,\{\_,\_\})$ a \emph{Poisson geometry} and will refer to the Poisson bracket as a \emph{Poisson structure}.

\subsection{Poisson structures as bivector fields}
\label{poissonstructurebivectorfields}

To the manifold $\Mcal$ there is associated the \emph{tangent sheaf} $\Tcal_\Mcal$ of derivations of $\Ocal_\Mcal$. Elements of $\Tcal_\Mcal$ are called \emph{vector fields} and it comes equipped with a Lie bracket
\[
 [A,B]:=A\circ B-B\circ A \bhs\bhs A,B\in\Tcal_\Mcal.
\]
Consider now the exterior algebra $\wedge^\bullet_{\Ocal_\Mcal}\Tcal$ of \emph{polyvector fields}. We will omit $\Ocal_\Mcal$ from the notation. 
It has a natural grading given by the tensor length, i.e.~$\wedge^i \Tcal_\Mcal$ are precisely the elements of degree $i$. The bracket of $\Tcal_\Mcal$ can be extended to a degree $-1$ Lie bracket on $\wedge^\bullet \Tcal_\Mcal$. The extended bracket
\begin{equation}
\label{oddschouten}
 [\_,\_]_{\tilde{S}}\colon\wedge^k \Tcal_\Mcal \wedge_\K \wedge^l \Tcal_\Mcal \to \wedge^{k+l-1} \Tcal_\Mcal
\end{equation}
is defined by 
\begin{multline*}
[A_1\wedge \dotsb \wedge A_k,B_1\wedge\dotsb\wedge B_l]_{\tilde{S}}:=\\
\sum_{i,j} (-1)^{i+j}[A_i,B_j]\wedge A_1 \wedge\dotsb\wedge \widehat{A_i} \wedge\dotsb\wedge A_k \wedge B_1 \wedge\dotsb\wedge \widehat{B_j} \wedge\dotsb\wedge B_k.
\end{multline*}
for $k\geq 1$, $l=0$, i.e~$B_0\in\Ocal_\Mcal$, by
\begin{align*}
&[A_1\wedge \dotsb \wedge A_k,B_0]_{\tilde{S}}:=
\sum_{i,j} (-1)^{i+k}A_i(B_0)\wedge A_1 \wedge\dotsb\wedge \widehat{A_i} \wedge\dotsb\wedge A_k,
\end{align*}
for $k=0$, $l\geq 1$ by
\begin{align*}
&[B_0,A_1\wedge \dotsb \wedge A_k]_{\tilde{S}}:=
\sum_{i,j} (-1)^{i}A_i(B_0)\wedge A_1 \wedge\dotsb\wedge \widehat{A_i} \wedge\dotsb\wedge A_k,
\end{align*}
and for $k=l=0$ by
\[
[A_0,B_0]_{\tilde{S}}:=0.
\]

Note that the above degree $-1$ Lie bracket on $\wedge^\bullet \Tcal_\Mcal$ is equivalent to the ordinary degree zero \emph{Schouten bracket} $[\_,\_]_{S}$ on $\wedge^\bullet \Tcal_\Mcal[1]$. We prefer to work with the former structure and will refer to it as the \emph{odd Schouten bracket}.

The \emph{cotangent sheaf} of a manifold $\Mcal$ is defined by  $\Omega_\Mcal^1:=\Hom_{\Ocal_\Mcal}(\Tcal_\Mcal, \Ocal_\Mcal)$ and  the \emph{de Rham algebra} by $\Omega_\Mcal^\bullet:=\wedge^\bullet\Omega_\Mcal^1$. Note that there is a natural isomorphism $\Omega_\Mcal^i\iso\Hom_{\Ocal_\Mcal}(\wedge^i \Tcal_\Mcal, \Ocal_\Mcal)$. The elements of $\Omega_\Mcal^i$ are called \emph{$i$-forms}. To an element $f\in\Ocal_\Mcal$ there is an associated 1-form $df$ defined by $df(A):=A(f)$, for a vector field $A$.

From a \emph{bivector field} $\G$, i.e.~an element of $\wedge^2 \Tcal_\Mcal$, one obtains an operation
\[
[\_,\_]_\G\colon\Ocal_{\Mcal} \otimes \Ocal_{\Mcal}\to \Ocal_{\Mcal}
\]
defined by
\[
[f,g]_\G:=\G df\wedge dg. 
\]
This operation satisfies properties \eqref{skewsymmetry} and \eqref{derproperty} of \S \ref{poissongeompar} since $\G$ is a bivector field. Conversely, any bilinear operation $\Ocal_{\Mcal} \wedge \Ocal_{\Mcal}\to \Ocal_{\Mcal}$ satisfying the properties \eqref{skewsymmetry} and \eqref{derproperty} can be described by a bivector field in this way. The condition that $[\_,\_]_\G$ satisfies the Jacobi identity is equivalent to $[\G,\G]_{\tilde{S}}=0$. Thus the following definition is equivalent to the one given in \S \ref{poissongeompar}.

\begin{defn*}
A \emph{Poisson structure} on a manifold $\Mcal$ is a polyvector field $\G\in\wedge^{2}\Tcal_\Mcal$ satisfying $[\G,\G]_{\tilde{S}}=0$.
\end{defn*}

\subsection{Generalized Poisson structures}

In fact one does not need to consider only the solutions of $[\G,\G]_S=0$ which are of degree two. One generalization of Poisson geometry is to $n$-ary Poisson brackets. For $n$ even, a polyvector field $\G\in\wedge^n\Tcal_\Mcal$ defines a \emph{generalized Poisson structure} if $[\G,\G]_S=0$. The associated $n$-ary Poisson bracket is defined analogously to the case $n=2$; for a polyvector field $\G\in\wedge^{n}\Tcal_\Mcal$ it is given by
\[
 \{f_1,\dotsc,f_n\}:=\G d f_1 \wedge \dotsb \wedge d f_n.
\]
The condition $[\G,\G]_{\tilde{S}}=0$ translates into a generalized Jacobi identity.
Notice that for polyvector fields of a non-graded manifold the expression $[\G,\G]_S$ identically vanishes for $n$ odd. It is possible to define a Poisson bracket with properties mimicking the classical case also for $n$ odd, but then the generalized Jacobi identity can not be expressed through the Schouten bracket. See e.g.~ e.g.~\cite{Azc'arraga1996} and \cite{Vinogradov1998} for more on $n$-ary Poisson brackets.

\subsection{Bi-Hamiltonian structures}
\label{compatiblepoisson}

Let $\Mcal$ be a manifold equipped with a pair of Poisson brackets $\{\_\hspace{1pt},\_\}_1$ and $\{\_\hspace{1pt},\_\}_2$. Consider the bracket defined by their sum 
\[
\{\_\hspace{1pt},\_\}:=\{\_\hspace{1pt},\_\}_1+\{\_\hspace{1pt},\_\}_2.
\]
It is obviously skew symmetric and it satisfies the Leibniz property, but it does not always satisfy the Jacobi identity. A pair of Poisson brackets $\{\_\hspace{1pt},\_\}_1$ and $\{\_\hspace{1pt},\_\}_2$ are called \emph{compatible} if their sum satisfies the Jacobi identity and thus itself is a Poisson bracket.

\begin{defn*}
A manifold equipped with two compatible Poisson structures is called \emph{bi-Hamiltonian} or a \emph{Poisson pair}. 
\end{defn*}

Let $\G_1$ and $\G_2$ be bivector fields corresponding to a pair of Poisson brackets, thus they satisfy $[\G_1,\G_1]_{\tilde{S}}=0$ and $[\G_2,\G_2]_{\tilde{S}}=0$, respectively. The compatibility of the Poisson brackets is equivalent to $[\G_1+\G_2,\G_1+\G_2]_{\tilde{S}}=0$ which in turn is equivalent to $[\G_1,\G_2]_{\tilde{S}}=0$. By introducing a formal parameter $\hslash$, the above conditions are together equivalent to
\[
 [\G_1+\G_2\hslash,\G_1+\G_2\hslash]_{\tilde{S}_\hslash}=0.
\]
Here the bracket is the linearization in $\hslash$ of the odd Schouten bracket.

\subsection{Poisson structures on formal graded manifolds}
\label{formalgradedpoisson}

We now turn our attention to graded manifolds. More accurately,
we will consider only formal graded manifolds, i.e.~manifolds consisting of a formal neighborhood of a single point, and the grading we consider is over $\Z$. A graded vector space $V$ can be naturally viewed as a formal graded manifold by considering a formal neighborhood of the origin. We denote the distinguished point by $0$. Let $\{e_a\}$ be a homogeneous basis of $V$, and denote the associated dual basis by  $\{t^a\}$, with grading $|t^a|=-|e_a|$. The structure sheaf of $V$ is given by $\Ocal_V:=\widehat{\odot^\bullet V^*}\iso\K\llbracket t^a\rrbracket$.

A \emph{graded Poisson bracket} on a formal graded manifold $V$ is a degree zero bilinear operation $\{\_,\_\}\colon \Ocal_{V} \wedge \Ocal_{V}\to \Ocal_{V}$ satisfying the properties
\begin{enumerate}
 \item
\label{gradedskewsymmetry}
  $\{f,g\}+(-1)^{fg}\{g,f\}$  \bhs\bhs(graded skew-symmetry)\vspace{2pt}
 \item
\label{gradedjacobiproperty}
  $(-1)^{fh}\{f,\{g,h\}\}+(-1)^{gf}\{g,\{h,f\}\}+(-1)^{hg}\{h,\{f,g\}\}$ \vspace{2pt} \\(graded Jacobi identity) \vspace{2pt}
 \item
\label{gradedderproperty}
 $\{f,gh\}=\{f,g\}h+(-1)^{fg}g\{f,h\}$  \bhs\bhs (Leibniz property of $\{f,\_\}$).
\end{enumerate}
The notation $(-1)^f$ is short for $(-1)^{|f|}$ and will be used from now on. We see that a graded Poisson structure is a graded Lie algebra on $\Ocal_V$ with the extra property that the Lie bracket is a graded derivation in each argument with respect to the graded commutative multiplication on $\Ocal_V$. %This notion of graded Poisson bracket can be found e.g.~in \cite{Kostant1977}, although the grading considered there is over $\Z_2$, so called super grading, where the reader also can find more details on (super) graded manifolds.

\subsection{Graded Poisson structures as bivector fields}
\label{gradedpoissonstructurebivectorfields}

The tangent sheaf $\Tcal_V$ of vector fields of $V$ is the $\Ocal_V$-module of derivations of $\Ocal_V$. The tangent sheaf is generated over $\Ocal_V$ by the derivations $\{\ddt{a}\}$ with $\frac{\p t^b}{\p t^a}=\delta_{a,b}$. We note that $|\ddt{a}|=-|t^a|$.
The sheaf of polyvector fields is defined as
\[
\wedge^\bullet\Tcal_V:=\odot^\bullet_{\Ocal_V}(\Tcal_V[-1]).
\]
We denote the generators $s\ddt{a}$ by $\psi_a$, where $s$ is a formal symbol of degree one. Thus $|\psi_a|=-|t^a|+1$. With this notation we have $\wedge^\bullet\Tcal_V\iso\K\llbracket t,\psi\rrbracket$. The degree of a homogeneous polyvector field 
\[
 \G=\G^{a_1\dots a_{i}}_{b_1\dots b_{j}} t^{b_1}\dots t^{b_{j}} \psi_{a_1}\dotsb\psi_{a_i}
\]
is given by 
\[
|\G|=|t^{b_1}|+\dotsb+|t^{b_{j}}|+ |\psi_{a_1}|\dotsb|\psi_{a_i}|.
\]
Note that $\pvfs$ also has the grading described in  \S \ref{poissonstructurebivectorfields};  we will refer to this grading as the \emph{weight} and to the former as the \emph{cohomological degree} or simply as the degree. When $V$ is concentrated in degree zero these gradings coincide.

We define the odd Schouten bracket by
\[
 [A,B]_{\tilde{S}}:=A\bullet B +(-1)^{|A||B|+|A|+|B|}B\bullet A
\]
where we use the notation
\[
A\bullet B := \frac{\p A}{\p \psi_a}\frac{\p B}{\p t^a}.
\]
Note that the with the above grading the Schouten bracket is a degree $-1$ (cohomological as well as weight) Lie bracket and if $V$ is concentrated in degree zero, then this definition coincides with \eqref{oddschouten}. An interpretation of graded Poisson structures in terms of bivector fields vanishing on the Schouten bracket, analogous to that of \S \ref{poissonstructurebivectorfields} can be found in \cite{Cantrijn1991}.

\begin{defn*}
A \emph{graded Poisson structure} on a formal graded manifold $V$ is an element $\G\in\bvfs$ of degree two satisfying $[\G,\G]_{\tilde{S}}=0$.
\end{defn*}

\begin{rem*}
That we require the bivector field to be of degree two ensures that the associated Poisson bracket is of degree zero.
 \end{rem*}

% We call a graded Poisson structure \emph{pointed} if it vanishes at the distinguished point of $V$. 

\subsection{Extended Poisson structures}
\label{formalextendedpoissonstructures}

When translated to differential geometry the prop profile of Poisson geometry, to be discussed in detail in \S \ref{poissonpropprofile}, can be interpreted \cite{Merkulov2006} as polyvector fields $\G$ with the properties 
\begin{enumerate}
\item\label{poissonspreadout} $\G\in\wedge^{\bullet\geq 1}\Tcal_V$,
\item\label{poissondegreetwo} $|\G|=2$,
\item\label{poissonschouten} $[\G,\G]_{\tilde{S}}=0$,
\item\label{pointedproperty} $\G|_{0}=0$.
\end{enumerate}

If we want such polyvector fields to generalize Poisson structures then Property \eqref{pointedproperty} is not desirable. We propose the following definition.

\begin{defn*}
Am \emph{extended Poisson structure} on a formal graded manifold $V$ is an element $\G\in\wedge^{\bullet\geq 1}\Tcal_V$
of degree two satisfying $[\G,\G]_{\tilde{S}}=0$.
\end{defn*}

We call an extended Poisson structure \emph{pointed} if it satisfies Property \eqref{pointedproperty}. By Remark \ref{nonpointedrem} the prop profile essentially describes all extended Poisson structures.

Note that if $V$ is concentrated in degree zero, then an extended Poisson structure is an ordinary Poisson structure on $V$, i.e.~in this case $\G\in\wedge^2\Tcal_V$. %, thus extended Poisson structures is a natural extension of classical Poisson structures to the graded setting.

\subsection{The family of brackets of an extended Poisson structure}
\label{familyofbrackets}

The cotangent sheaf of a formal graded manifold $V$ is defined by $\Omega_V^1:=\Hom_{\Ocal_V}(\Tcal_V,\Ocal_V)$ and the de Rham algebra by
\[
\Omega^\bullet_V:=\odot^\bullet_{\Ocal_V}(\Omega^1_V[1]).
\]
A basis over $\Ocal_V$ of $\Omega_V^1$ is given by $\{dt^{a}\}$, where $\psi^b dt^a:= dt^a(\psi_b)=\delta_{a,b}$ and $|dt^a|=|t^a|-1$.

To a polyvector field $\G=\sum_{n\geq 1} \G_n$, with $\G_n:=\G^{a_1\dots a_{n}}(t)\psi_{a_1}\dotsb\psi_{a_n}$, we associate a family of brackets as follows. We define an $n$-ary bracket $L_n\colon\otimes^n\Ocal_V\to\Ocal_V$  by
\begin{align*}
 L_n(f_1,\dotsc,f_n):&=\G_n d f_1 \wedge \dotsb \wedge d f_n\\
&=(-1)^{\epsilon}\G^{a_1\dots a_{n}}(t)(\p_{a_1}f_1)\dotsb(\p_{a_n}f_n).
\end{align*}
Here the sign $(-1)^\epsilon$ is given by
\[
 \epsilon= a_n(f_1+\dotsb+f_{n-1}+n-1)+(a_{n-1})(f_1+\dotsb+f_{n-2}+n-2)+\dotsb+a_2(f_1+1).
\]

\begin{defn*}
 A vector space $V$ together with a family $\{l_n\}_{n\in\N}$ of graded skew symmetric maps $l_n\colon \otimes^n V\to V$ of degree $2-n$ is called an \emph{$\ELL_\infty$ algebra} if the following condition is satisfied for all $n\in\N$
\begin{equation}
\label{Linfty}
 \hspace{-10pt}\sum_{\substack{r+s=n+1\\ \text{$(s,n-s)$-unshuffles $\sigma$}}}  \epsilon(\sigma)\sgn(\sigma)(-1)^{r(s-1)} l_r(l_s(v_{\sigma(1)},\dotsc v_{\sigma(s)}),v_{\sigma(s+1)},\dotsc, v_{\sigma(n)}).
\end{equation}
Here the sign $\epsilon(\sigma)$ is the sign appearing from the Koszul-Quillen sign rule.
% \[
%  u_1\wedge\dotsb\wedge u_i = \epsilon(\sigma) \sgn(\sigma) u_{\sigma(1)}\wedge\dotsb\wedge u_{\sigma(i)}
% \]
\end{defn*}

\begin{thm}
\label{poissonLinftythm}
The brackets $L_n$ associated to a polyvector field $\G\in\wedge^{\bullet\geq 1}\Tcal_V$ as above are graded skew commutative and have the graded Leibniz property in each argument, i.e.~for all $n\geq 1$ and all $2\leq j \leq n$
\begin{multline*}
 L_n(f_1,\dotsc, f_{j-1},gh ,f_j,\dotsc f_n)=\\(-1)g L_n(f_1,\dotsc, f_{j-1},h ,f_j,\dotsc f_n)+(-1)L_n(f_1,\dotsc, f_{j-1},g ,f_j,\dotsc f_n)h.
\end{multline*}
Moreover, the family of brackets $\{ L_n \}_{n\geq 1}$ gives $\Ocal_V$ the structure of $\ELL_\infty$ algebra if and only if $\G$ is 
an extended Poisson structure. %,i.e~ if $\G$ is of degree two and satisfies $[\G,\G]_{\tilde{S}}=0$.
\end{thm}

\begin{proof}
That the brackets $L_n$ are graded skew symmetric is immediate from the definition. The Leibniz property is satisfied since $L_n(f_1,\dotsc, f_{j-1},\hspace{2pt}\_\hspace{1pt} ,f_j,\dotsc f_n)$ is a vector field.
We notice that
\begin{align*}
 |L_n|=|\G^{a_1\dots a_{n}}(t)|+(|\p_{a_1}|+\dotsb+|\p_{a_n}|)=|\G_n|-n=2-n.
\end{align*}
Thus $\G$ is of degree two if and only if $L_n$ is of degree $2-n$.
For the Poisson bracket associated to a bivector field $P$ the condition $[P,P]_{\tilde{S}}=0$ is equivalent to the Poisson bracket satisfying the Jacobi identity. That the $L_i$ satisfy the $\ELL_\infty$-conditions is proven much in the same way. It is a tedious but straightforward computation to verify that the brackets $L_n$ associated to a polyvector field $\G$ of degree two satisfies the equation \eqref{Linfty} if and only if $[\G,\G]_{\tilde{S}}=0$.
\end{proof}

This leads to another definition of extended Poisson structures on formal graded manifolds, which by the preceding theorem is equivalent to the one we gave in \S \ref{formalextendedpoissonstructures}.

\begin{defn*}
An \emph{extended Poisson structure} on a formal graded manifold $V$ is an $\ELL_\infty$ algebra $\{ L_n \}_{n\geq 1}$ on $\Ocal_V$ such that the brackets $L_n$ have the Leibniz property in each argument.
\end{defn*}

\subsection{Graded bi-Hamiltonian structures}
\label{gradedbihamstructure}

A graded Bi-Hamiltonian structure on a formal manifold $V$ is defined analogously to the non-graded case; it is a pair $\G_1$ and $\G_2$ of graded Poisson structures such that their sum $\G_1+\G_2$ again is a graded Poisson structure. In particular this implies that the associated Poisson brackets are a pair of compatible graded Lie brackets. In Section \ref{geominterpretation} we propose a definition of extended bi-Hamiltonian structures, obtained from the machinery of prop profiles, such that the associated family of brackets is the strongly homotopy structure associated to a pair of compatible Lie brackets.
\section{Prop profiles I: Extracting the prop}
\label{propprofileextraction}

Finding a prop profile of a geometric structure is done in two main steps. First one extracts the fundamental part of the geometric structure and encodes it as a prop. Then one computes a minimal resolution of the extracted prop. The aim of this section is to extract the prop of bi-Hamiltonian structures. We begin by recalling the prop profile of Poisson structures originally constructed in \cite{Merkulov2006}.

\subsection{The prop profile of Poisson structures}
\label{poissonpropprofile}

Consider the formal graded manifold associated to a vector space $V$. Recall that a Poisson structure on $V$ is a degree two bivector field $P\in\wedge^\bullet\Tcal_V$ satisfying $[P,P]_{\tilde{S}}=0$. To be precise we consider a pointed Poisson structure. With the notation of the previous section we have
\[
 P=\sum_{n\geq 1}P^{a_1 a_2}_{b_1\dotsb b_n}t^{b_1}\dotsb t^{b_n}\psi_{a_1}\psi_{a_2}.
\]
We can interpret this as a collection of degree zero maps
\[
p_n\colon \odot^n V \to \wedge^2 V
\]
defined by
\[
 p_n(e_{b_1}\odot\dotsb\odot e_{b_n})\to P^{a_1 a_2}_{b_1\dotsb b_n} e_{a_1}\wedge e_{a_2}.
\]
The condition $[P,P]_{\tilde{S}}=0$ then translates into a sequence of quadratic relations of these maps. Merkulov's idea \cite{Merkulov2006} was that this algebraic structure corresponds to just the degree zero part of the resolution of a prop. This means that a certain part of the structure is fundamental and the rest of the maps are higher homotopies, many of which may not be visible in degree zero.

Kontsevich in \cite{Kontsevich2003} gave an interpretation of degree two (degree one if we consider $\wedge^\bullet\Tcal_V[1]$) vector fields $Q$ satisfying $[Q,Q]_{\tilde{S}}=0$ as $\ELL_\infty$ algebras. A vector field $Q$ given by
\[
 Q=\sum_{n\in\N}Q^a_{b_1\dotsb b_n}t^{b_1}\dotsb t^{b_n}\psi_{a}
\]
gives rise to a family of degree one maps
\[
q_n\colon \odot^n V \to V
\]
defined by
\[
 q_n(e_{b_1}\odot\dotsb\odot e_{b_n})\to Q^a_{b_1\dotsb b_n} e_a.
\]
The vector field satisfies $[Q,Q]_{\tilde{S}}=0$ if and only if the maps $q_n$ give $V[-1]$ the structure of $\ELL_\infty$ algebra. Such vector fields are called homological. In fact a homological vector field corresponds to a part of the prop profile of Poisson structures.  The reason why it does not show in a classical Poisson structure is that it lies in the wrong degree; degree one maps vanish on a vector space concentrated in degree zero. The fundamental part of the $\ELL_\infty$ structure obtained from $Q$ is the map $q_2$, the rest of the $q_n$ are higher homotopies. We denote the corresponding part of the vector field by $\hat{Q}$. The properties of a Poisson structure is modeled by the vanishing of the Schouten bracket, in particular the condition $[Q,Q]_{\tilde{S}}=0$ implies
\begin{equation}
\label{Qcondition}
[\Qf,\Qf]_{\tilde{S}}=0
\end{equation}
which is equivalent to $q_2$ being a degree one Lie bracket on $V$. The maps $\{q_n\}$ satisfying conditions dictated by $[Q,Q]_{\tilde{S}}=0$ give $V$ the structure obtained from the minimal resolution $(\Lieone)_\infty$ of the operad $\Lieone$ of Lie algebras with the bracket of degree one. Identifying $q_2$ with $\Ysmall$ the operadic interpretation of \eqref{Qcondition} is
\begin{equation}
\label{Qrelation}
 \YY{1}{2}{3}+\YY{2}{3}{1}+\YY{3}{1}{2}=0.
\end{equation}

Another fundamental part of the data of a Poisson structure is $p_1$. We denote the part of $P$ corresponding to $p_1$ by $\Pf$. In contrast to $\Qf$ it lies in a degree where one can spot it in the case of classical Poisson structures, but for degree reasons no part corresponding to the higher homotopies of $p_1$ is visible. The condition
\begin{equation}
\label{Pcondition}
[\Pf,\Pf]_{\tilde{S}}=0
\end{equation}
is equivalent to the map $p_1$ defining a Lie coalgebra structure on $V$. Identifying $p_1$ with $\coYsmall$, the propic (though still operadic in its nature) depiction of \eqref{Pcondition} is
\begin{equation}
\label{Prelation}
 \coYcoY{1}{2}{3}+\coYcoY{2}{3}{1}+\coYcoY{3}{1}{2}=0.
\end{equation}

To obtain the maps $p_n$ with $n\geq 2$ we need to combine the fundamental parts $\Pf$ and $\Qf$. Their relation is also modeled by the Schouten bracket
\begin{equation}
\label{PQcondition}
[\Pf,\Qf]_{\tilde{S}}=0,
\end{equation}
which translates to
\begin{equation}
\label{PQrelation}
 \YcoY{1}{2}{1}{2}-\coYopY{1}{2}{1}{2}+\coYopY{2}{1}{1}{2}-\coYopY{1}{2}{2}{1}+\coYopY{2}{1}{2}{1}=0.
\end{equation}
We can simultaneously express the conditions \eqref{Qcondition}, \eqref{Pcondition}, and \eqref{PQcondition} by
\[
 [\Pf+\Qf,\Pf+\Qf]_{\tilde{S}}=0.
\]
To describe Poisson geometry as a minimal resolution of an algebraic object we need to go beyond operads; since $p_1$ has multiple outputs and $q_2$ multiple inputs we need a prop to model them. 

\begin{defn*}
The prop $\Poisson$ is the quadratic prop $\QO{M}{R}$ where $M$ is the $\sbimodule$ given by $M(1,2)=\K\Ysmall=\bbone_1\otimes\bbone_2[-1]$, $M(2,1)=\K\coYsmall=\sgn_2\otimes\bbone_1$, and zero for other $(m,n)$. The relations $R$ are given by \eqref{Qrelation}, \eqref{Prelation}, and \eqref{PQrelation}. %Algebras over $\Poisson$ are called \emph{\Poissonalgs}.
\end{defn*}

\begin{rems*}\hspace{1pt}\vspace{-15pt}\newline
\begin{enumerate}
 \item This prop is similar to the prop $\Liebi$ of Example \ref{liebi} with the difference being that the bracket and cobracket lie in degrees differing by one, explaining the $1$ in the notation.
 \item Actually, since the relations are dioperadic and constitute what is called a distributive law, cf.~\ref{distributivelaw}, it suffices to encode the fundamental part of the geometric structure as a dioperad. Its resolution is then easier to compute and is straightforwardly extended to a resolution of the corresponding prop.
\end{enumerate}
\end{rems*}

Merkulov called the generators and relations of $\Poisson$ the genes and engineering rules of Poisson geometry, together constituting its genome. By computing its minimal resolution $\Poisson_\infty$ explicitly and translating representations of it into polyvector fields he obtained the following result.

\begin{thm}[Proposition 1.5.1 of \cite{Merkulov2006}]
\label{merkulovmainthm}
There is a one-to-one correspondence between representations of $\Poisson_\infty$ in a dg vector space $V$ and pointed extended Poisson structures on the formal manifold associated to $V$. % vanishing at the distinguished point.
\end{thm}

To be precise, the above theorem holds if we consider the differential of the vector space $V$ to be part of the data of a representation. This will be explained in detail in the case of bi-Hamiltonian structures.

\begin{rem}
\label{nonpointedrem}
That the Poisson structures considered in Theorem \ref{merkulovmainthm} are pointed, i.e.~vanish at the distinguished point, poses no real problem. Given an arbitrary non-pointed Poisson structure on a formal graded manifold $V$, i.e.~an element $\G\in\wedge^\bullet\Tcal_V$ such that $[\G,\G]_{\tilde{S}}$ and $\G|_0\neq 0$, it can be obtained from $\Poisson_\infty$ by considering representations in $V\oplus\K$. For a formal variable $x$, viewed as a coordinate on $\K$, we have that $x\G\in\wedge^\bullet\Tcal_{V \oplus \K}$ vanishes at the distinguished point of $V\oplus \K$ and since $x\G$ still satisfies $[x\G,x\G]_{\tilde{S}}=0$, it corresponds to a representation of $\Poisson_\infty$.
\end{rem}

\subsection{Extracting the prop of bi-Hamiltonian structures}

A bi-Hamiltonian structure on the formal manifold associated to a vector space $V$ is a pair of bivector fields $P_1$ and $P_2$ satisfying $[P_1,P_2]_{\tilde{S}}=0$, $[P_2,P_2]_{\tilde{S}}=0$, and $[P_1,P_2]_{\tilde{S}}=0$. We want again to extract a prop encoding the fundamental part of this structure. As in the previous paragraph we let $\Pfa$ and $\Pfb$ denote the parts of $P_1$ and $P_2$ corresponding to maps $V \to \wedge^2 V$. The conditions
\begin{equation}
\label{Pabcondition}
[\Pfa,\Pfa]_{\tilde{S}}=0 \bhs \text{and} \bhs[\Pfb,\Pfb]_{\tilde{S}}=0
\end{equation}
are equivalent to that the maps corresponding to $\Pfa$ and $\Pfb$ each give $V$ the structure of Lie coalgebra.

\begin{defn*}
Let $V$ be a vector space and let $\Delta_1$ and $\Delta_2$ be Lie cobrackets on $V$. We say that the cobrackets are \emph{compatible} if their sum $\Delta_1+\Delta_2$ again is a Lie cobracket. We denote the quadratic prop encoding this structure by $\Colietwo$.
\end{defn*}

We depict the maps corresponding to $\Pfa$ and $\Pfb$ with $\coWsmall$ and $\coBsmall$, respectively. The compatibility condition $[\Pfa,\Pfb]_{\tilde{S}}=0$
can then be illustrated by
\[
 \coWcoB{1}{2}{3}+\coWcoB{2}{3}{1}+\coWcoB{3}{1}{2}+\coBcoW{1}{2}{3}+\coBcoW{2}{3}{1}+\coBcoW{3}{1}{2}=0,
\]
which means that the pair $(\Pfa,\Pfb)$ gives $V$ the structure of compatible Lie coalgebras.

We have a similar definition of compatible Lie algebras. 

\begin{defn*}
Let $V$ be a vector space and let $[\_,\_]_1$ and $[\_,\_]_2$ be Lie brackets on $V$. We say that the brackets are \emph{compatible} if their sum $[\_,\_]_1+[\_,\_]_2$ again is a Lie bracket. We denote the quadratic operad encoding this structure by $\Lietwo$.
\end{defn*}

The operad $\Lietwo$ was defined in \cite{Dotsenko2007}. Note that $\Colietwo$ and $\Lietwo$ differ only in the orientation of the defining graphs.

From the experience of constructing the prop profile of Poisson structures we expect a homological vector field $Q$ compatible with both $P_1$ and $P_2$ to be present, i.e.~satisfying $[P_1,Q]_{\tilde{S}}=0$ and $[P_2,Q]_{\tilde{S}}=0$. The compatibility of the fundamental part $\Qf$ with $\Pfa$ and $\Pfb$ means that the maps corresponding to the pairs $(\Pfa,\Qf)$ and $(\Pfb,\Qf)$ both give $V$ the structure of $\Poisson$ algebra.

To express these conditions with a single equation we introduce a formal parameter $\hslash$. The conditions
\[
[\Qf,\Qf]_{\tilde{S}}=0,\shs [\Pfa,\Qf]_{\tilde{S}}=0,\shs [\Pfa,\Pfa]_{\tilde{S}}=0,\shs [\Pfa,\Pfa]_{\tilde{S}}=0,\shs \shs\text{and}\shs [\Pfa,\Pfb]_{\tilde{S}}=0
\]
are then all subsumed by
\begin{equation}
\label{bipoissondioperadcondition}
 [\Qf+\Pfa+\Pfb\hslash,\Qf+\Pfa+\Pfb\hslash]_{S_\hslash}=0.
\end{equation}
Here the bracket is the linearization in $\hslash$ of the Schouten bracket. As in the case of Poisson structures, the relations \eqref{bipoissondioperadcondition} are dioperadic and in order to make easier the computation of the resolution of the corresponding prop we extract the dioperad encoding these relations.

\begin{defn*}
We define the quadratic dioperad $\Bipoisson$ by
\[
\Bipoisson=\QO{M}{R}.
\]
Here $M=\{M(m,n)\}_{m,n\geq 1}$ is the $\sbimodule$
\[
M(m,n)=
\begin{cases}
\bbone_1\otimes(\bbone_2[-1]) & \text{ if $(m,n)=(1,2)$}\\
(\sgn_2\oplus\sgn_2) \otimes \bbone_1 & \text{ if $(m,n)=(2,1)$} \\
0 & \text{otherwise}
\end{cases}.
\]
Denote a (1,2)-graph decorated with the natural basis element of $M(1,2)$ by $\Ysmall$ and (1,2)-graphs decorated with the basis elements of $M(2,1)$ by $\coWsmall$ and $\coBsmall$. The relations $R=R(1,3)\sqcup R(3,1) \sqcup R(2,2)$ consists of the following subsets $R(i,j)\subset\Free_{(2)}(M)(i,j)$
\begin{align}
R(1,3): &\YY{1}{2}{3}+\YY{2}{3}{1}+\YY{3}{1}{2}    \label{jacobi} \\%&\text{(Jacobi identity)}
R(3,1): &\coWcoW{1}{2}{3}+\coWcoW{2}{3}{1}+\coWcoW{3}{1}{2}, \shs \coBcoB{1}{2}{3}+\coBcoB{2}{3}{1}+\coBcoB{3}{1}{2}, \label{cojacobi}\\%& \text{(Co-Jacobi identities)}
        &\coWcoB{1}{2}{3}+\coWcoB{2}{3}{1}+\coWcoB{3}{1}{2}+\coBcoW{1}{2}{3}+\coBcoW{2}{3}{1}+\coBcoW{3}{1}{2} \label{coliecompatibility}\\ %&\text{($\Colietwo$ compatibility)}
R(2,2): &\YcoW{1}{2}{1}{2}-\coWopY{1}{2}{1}{2}+\coWopY{2}{1}{1}{2}-\coWopY{1}{2}{2}{1}+\coWopY{2}{1}{2}{1}, \label{lieonebicompatibilitya}\\
        &\YcoB{1}{2}{1}{2}-\coBopY{1}{2}{1}{2}+\coBopY{2}{1}{1}{2}-\coBopY{1}{2}{2}{1}+\coBopY{2}{1}{2}{1}.
\label{lieonebicompatibilityb}
\end{align}
\end{defn*}

By this we have obtained the genes and engineering rules of bi-Hamiltonian structures, the genetic code. We are now ready to plug them into the machinery of Koszul resolutions.
\section{Prop profiles II: Computing the resolution}
\label{propprofileresolution}

In this section we will compute the minimal resolution of the prop associated to the dioperad $\Bipoisson$ constructed in the previous section. This is done by first computing the dioperadic resolution and then extending it to a propic resolution. Recall from \S \ref{koszulresolution} that one way of obtaining a resolution of a Koszul dioperad $\Pcal$ is by computing the Koszul dual codioperad $\Pcal^{\antishriek}$ and then apply the cobar construction, i.e.~$\Omega(\Pcal^{\antishriek})\isoto \Pcal$. The differential of this resolution is determined by the cocomposition product of $\Pcal^{\antishriek}$. This codioperad as well as its cocomposition product can most readily be obtained by computing the Koszul dual dioperad $\Pcal^{!}$ and then consider its linear dual. We begin by presenting a tool for showing Koszulness.

\subsection{Distributive laws}
\label{distributivelaw}

From a quadratic dioperad one can extract two operads. First we note that to a dioperad $\Pcal$ one can associate its \emph{opposite dioperad} defined by $\Pcal^{\op}(m,n):=\Pcal(n,m)$. The composition product $\mu^{\op}$ is obtained from $\mu$ by reversing the direction of all graphs. Thus to a quadratic dioperad $\Pcal$ we can associate two operads $\Pcal_U$ and $\Pcal_D$ defined by $\Pcal_U(n):=\Pcal(1,n)$ and $\Pcal_D(n):=\Pcal^{\op}(1,n)$. Explicitly, for a quadratic dioperad $\Pcal=\QO{M}{R}$ with $M$ concentrated in $M(1,2)$ and $M(2,1)$, we have
\[
\Pcal_U=\QO{M(1,2)}{R(1,3)}, \bhs \bhs \Pcal_D=\QO{M(2,1)^{\op}}{R(3,1)^{\op}},
\]
where $R(1,3)$ is the part of $R$ in $\Free_{(2)}(M)(1,3)$, $M(2,1)^{\op}$ is the $\s$-module given by $M(2,1)^{\op}(2)=M(2,1)$ and zero otherwise, and $R(3,1)^{\op}$ are the relations in $\Free_{(2)}(M(2,1)^{\op})$ obtained from $R(3,1)\subset\Free_{(2)}(M)(3,1)$ by reversing the direction of the decorated graphs.

We also note that to any operad $\Pcal$ one can associate a dioperad $\widetilde{\Pcal}$ defined by $\widetilde{\Pcal}(1,n):=\Pcal(n)$ and $\widetilde{\Pcal}(m,n)=0$ for $m\neq 1$.

Next we define a product of dioperads introduced in \cite{Gan2003}. We define a \emph{two-level graph} to be a graph such that any vertex is connected to at least one other vertex and is connected to other vertices either only via its output edges or only via its input edges. The vertices can thus be divided into two levels in a unique way. We say that the vertices only connected via their outputs lie on the upper level and that the vertices only connected via their inputs lie on the lower level. Further, we call a graph $G$ \emph{reduced} if for all $v\in V_G$ it is true that $|\inp_v|\geq 1|$, $|\outp_v|\geq 1|$, and $|\outp_v|+|\inp_v\geq 3|$.   Let $\Pcal$ and $\Qcal$ be dioperads, we then define
\[
 \Pcal\Box\Qcal:=\bigoplus_{G\in\Gfrak^{\downarrow,2}_{c,0}}G\langle \Pcal^{V_1},\Qcal^{V_2} \rangle,
\]
where $\Gfrak^{\downarrow,2}_{c,0}$ is the subfamily of $\Gfrak^{\downarrow}_{c,0}$ consisting of reduced two-level graphs and $V_1$ and $V_2$ are the vertices on the lower and upper level, respectively. We say that a quadratic dioperad $\Pcal$ is given by a \emph{distributive law} if $\widetilde{\Pcal}_U\Box(\widetilde{\Pcal}_D)^{\op}=\Pcal$.

The following theorem was proved by W.~Gan.
\begin{thm}[Theorem 5.9 of \cite{Gan2003}]
\label{ganthm}
Let $\Pcal$ be a quadratic operad. If $\Pcal_U$ and $\Pcal_D$ are Koszul operads and $\widetilde{\Pcal}_U\Box(\widetilde{\Pcal}_D)^{\op}(i,j)=\Pcal(i,j)$ for $(i,j)=(2,2),(2,3),(3,2)$, then $\widetilde{\Pcal}_U\Box(\widetilde{\Pcal}_D)^{\op}=\Pcal$ and $\Pcal$ is Koszul.
\end{thm}

See \cite{Gan2003}, \cite{Markl1996}, and \cite{Dotsenko2007a} for details on distributive laws.

\subsection{Koszulness of $\Bipoisson$}

Now we are ready to show the following result.

\begin{prop}
The dioperad $\Bipoisson$ and the properad generated by it are Koszul.
\end{prop}
\begin{proof}
We observe $\Bipoisson_U=\Lieone$ and $\Bipoisson_D=\Lietwo$. It was shown in \cite{Strohmayer2007a} that the operad $\Lietwo$ is Koszul and it is well-known that $\Lieone$ is Koszul. It is straightforward to see that $\Lieone\Box(\Lietwo)^{\op}(i,j)=\Bipoisson(i,j)$ for $(i,j)=(2,2),(2,3),(3,2)$; thus by Theorem \ref{ganthm} we obtain that $\Bipoisson$ is Koszul.
\end{proof}

\subsection{The Koszul dual dioperad of $\Bipoisson$}

\begin{prop}
The Koszul dual of $\Bipoisson$ is
\[
\Bipoisson^!=\QO{N}{S}
\]
where
\[
N(m,n)=
\begin{cases}
\bbone_1\otimes(\sgn_2[1]) & \text{ if $(m,n)=(1,2)$}\\
(\bbone_2\oplus\bbone_2) \otimes \bbone_1 & \text{ if $(m,n)=(2,1)$} \\
0 & \text{otherwise}
\end{cases}
\]
and $S=S(1,3)\sqcup S(3,1) \sqcup S(2,2)$ consists of subsets $S(i,j)\subset\Free(N)(i,j)$. If we denote the natural basis element of $N(1,2)$ by $\Ysmall$ and the basis elements of $N(1,2)$ by $\Wsmall$ and $\Bsmall$,
then $S$ is given by
\begin{align}
S(1,3): &\YY{1}{2}{3}-\YY{2}{3}{1}, \YY{1}{2}{3}-\YY{3}{1}{2} \label{associativity}\\
S(3,1): &\coWcoW{1}{2}{3}-\coWcoW{2}{3}{1}, \shs \coWcoW{1}{2}{3}-\coWcoW{3}{1}{2} \label{coassociativitya}\\
              &\coBcoB{1}{2}{3}-\coBcoB{2}{3}{1},\shs \coBcoB{1}{2}{3}-\coBcoB{3}{1}{2}, \label{coassociativityb}\\
              &\coWcoB{1}{2}{3}-\coWcoB{2}{3}{1}, \shs \coWcoB{1}{2}{3}-\coWcoB{3}{1}{2},\label{comcompatibilitya}\\
              &\coWcoB{1}{2}{3}- \coBcoW{1}{2}{3}, \shs \coWcoB{2}{3}{1}- \coBcoW{2}{3}{1}, \shs \coWcoB{3}{1}{2}- \coBcoW{3}{1}{2}\label{comcompatibilityb}\\
S(2,2): &\YcoW{1}{2}{1}{2}+\coWopY{1}{2}{1}{2},\shs \YcoW{1}{2}{1}{2}-\coWopY{2}{1}{1}{2},\shs
         \YcoW{1}{2}{1}{2}+\coWopY{1}{2}{2}{1},\shs \YcoW{1}{2}{1}{2}-\coWopY{2}{1}{2}{1}\label{bicomcompatibilitya}\\
        &\YcoB{1}{2}{1}{2}+\coBopY{1}{2}{1}{2},\shs \YcoB{1}{2}{1}{2}-\coBopY{2}{1}{1}{2}, \shs
         \YcoB{1}{2}{1}{2}+\coBopY{1}{2}{2}{1},\shs \YcoB{1}{2}{1}{2}-\coBopY{2}{1}{2}{1}.\label{bicomcompatibilityb}
\end{align}

\end{prop}
\begin{proof}
For $\Bipoisson=\QO{M}{R}$ we first observe that $N=M^{\cz}$. Recalling the pairing described in \ref{koszuldualGalgebra} we notice that $(S)$ is the orthogonal complement to $(R)$ with respect to this pairing. %By direct calculation.
\end{proof}

Like $\Bipoisson$, its Koszul dual dioperad $\Bipoisson^!$ is built from two operads. The first one, $(\Bipoisson^!)_U$, generated by $\Ysmall$ with relations \eqref{associativity} is the operad $\Comone$ of commutative algebras with the operation of degree minus one, Koszul dual to $\Lieone$. The second one, $(\Bipoisson^!)_D$, generated by $\coWsmall$ and $\coBsmall$ with relations \eqref{coassociativitya}, \eqref{coassociativityb}, \eqref{comcompatibilitya} and \eqref{comcompatibilityb} is the operad of totally compatible commutative algebras. This operad was defined and shown to be Koszul dual to $\Lietwo$ in \cite{Dotsenko2007}. See \cite{Strohmayer2007a} for a treatment of operads encoding compatible structures. The relations \eqref{bicomcompatibilitya} and \eqref{bicomcompatibilityb} are orthogonal to the compatibility relations of Lie $1$-bialgebras and are related to the dioperad of Frobenius algebras; the dioperad of Frobenius algebras is Koszul dual to the dioperad of Lie bialgebras, see e.g.~\cite{Gan2003}.

By straightforward graph calculations we obtain the following result.

\begin{prop}
\label{koszuldualbasis}
The dioperad $\Bipoisson^!$ has as underlying $\s$-bimodule
\[
\Bipoisson^!(m,n)=
\begin{cases}
\underbrace{( \bbone_m\oplus \dotsb \oplus\bbone_m )}_{\text{$m$ copies}} \otimes \sgn_n [n-1]  & \text{if $m+n\geq3$}\\
0 & \text{otherwise}.
\end{cases}
\]
Explicitly, a $\K$-basis for $\Bipoisson^!(m,n)$ is given by
\[
\left\{ \shs \YcoWcoBcorolla \shs \right\}_{0\leq i \leq m-1}.
\]
\end{prop}

\subsection{The minimal resolution of $\Bipoisson$}
\label{bipoissonresolution}

We now have everything we need to describe a minimal resolution of $\Bipoisson$ explicitly.

\begin{thm}
\label{bipoissonresolutionthm}
The Koszul resolution $\Bipoisson_\infty$ of the dioperad $\Bipoisson$ is the quasi-free dioperad on the $\sbimodule$  $E=\{E(m,n)\}_{m,n\geq 1, m+n\geq 3}$ where
\[
E(m,n)=
\begin{cases}
\underbrace{( \sgn_m\oplus \dotsb \oplus\sgn_m )}_{\text{$m$ copies}} \otimes \bbone_n [m-2]  & \text{if $m+n\geq3$}\\
0 & \text{otherwise}.
\end{cases}
\]
We denote the element of $E$ corresponding to the basis element of $\Bipoisson^!(m,n)$ with $i$ black operations by
\[
\boxcorolla{1}{n}{i}{1}{m}\bhs\sim\bhs\YcoWcoBcorolla \shs.
\]

The differential of $\Bipoisson_\infty$ is then given by
\[
\delta \colon\boxcorolla{1}{n}{i}{1}{m} \mapsto \sum_{\substack{ 1 \leq k \leq n \\ 0 \leq j \leq m-1 \\ 2 \leq j+k \leq m+n-2 \\ i_1+i_2=i \\ (k,n-k)\text{-shuffles} \shs \tau \\ (j,m-j)\text{-shuffles} \shs \sigma }}
\hspace{-21pt}(-1)^{\sgn(\sigma)+j(m-j)}\hspace{5pt}
\mirroredsplitboxcorolla{ \sigma(j+1) }{ \sigma(m) }{ \tau(k+1) }{ \tau(n)}{\sigma(1) }{ \sigma(j) }{ \tau(1) }{ \tau(k) }.
\]
\end{thm}
\begin{proof}

From the Koszulness of $\Bipoisson$ it follows that $\Bipoisson_\infty=\Omega(\Bipoisson^{\antishriek})$ is a quasi-free resolution of $\Bipoisson$. The cobar construction is given by $\Omega(\Bipoisson^{\antishriek})=\Free(\Sigma\ol{\Bipoisson^{\antishriek}}))$. We observed in \S \ref{koszuldualGalgebra} that for a dioperad $\Pcal$ we have $(\Pcal^{\antishriek})_{(s)}(m,n)\iso\Sigma^{-s}((\Pcal^!)_{(s)}(m,n))^{\cz}$. Since $\Bipoisson^{\antishriek}(m,n)$ is concentrated in weight $m+n-2$ it follows from Proposition \ref{koszuldualbasis} that
\[
\Bipoisson^{\antishriek}(m,n)=\underbrace{( \sgn_m\oplus \dotsb \oplus\sgn_m )}_{\text{$m$ copies}} \otimes \bbone_n [m-1].
\]
Setting $E=\Sigma\ol{\Bipoisson^{\antishriek}}$ the first assertion of the theorem follows.

Since $\Bipoisson$ has zero differential it follows that the differential $\delta$ of $\Omega(\Bipoisson^{\antishriek})$ is fully determined by the cocomposition coproduct of $\Bipoisson^{\antishriek}$. Through tedious but straightforward graph calculations one can determine the composition product of $\Bipoisson^{!}$. Considering the linear dual of this product yields the differential $\delta$.
\end{proof}

\subsection{The minimal resolution of $\Lietwo$}
\label{lietworesolution}

The minimal resolution of the operad $\Lietwo$ of compatible Lie algebras will play an important role in the interpretation of bi-Hamiltonian structures on formal graded manifolds as algebraic structures on the structure sheaf. We get it for free from the preceding theorem.

\begin{cor}
The minimal resolution $(\Lietwo)_\infty$ of the operad $\Lietwo$ of pairs of linearly compatible Lie algebras is the quasi-free operad on the $\smodule$  $E=\{E(n)\}_{n\geq2}$ where
\[
E(n)=
\begin{cases}
\underbrace{\sgn_n\oplus \dotsb \oplus\sgn_n }_{\text{$n$ copies}} [n-2]  & \text{if $n\geq2$}\\
0 & \text{otherwise}.
\end{cases}
\]
Denote the natural basis of $E(n)$ (cf.~Theorem \ref{bipoissonresolutionthm}) by
\[
\lietwoboxcorolla{1}{n}{i} \bhs\bhs 0\leq i \leq n-1.
\]
The differential of $(\Lietwo)_\infty$ is then given by
\[
\delta\colon\lietwoboxcorolla{1}{n}{i}\mapsto
 \sum_{\substack{ 2 \leq k \leq n-1 \\ i_1+i_2=i \\ (k,n-k)\text{-shuffles} \shs \tau }}
\hspace{-15pt}(-1)^{\sgn(\tau)+(k-1)(n-k+1)}\hspace{-5pt}\lietwosplitboxcorolla{\tau(1)}{\tau(k)}{\tau(k+1)}{\tau(n)}.
\]
\end{cor}

Algebras over the operad $(\Lietwo)_\infty$ are defined as follows.

\begin{defn*}
 A dg vector space $V$ together with a family $\{\leftsub{i}{l}_n\}_{n\in\N,1\leq i \leq n}$ of maps $\leftsub{i}{l}_n\colon \wedge^n V\to V$ of degree $2-n$ is called an \emph{$\ELL^2_\infty$-algebra} if the following condition is satisfied for all $n,k\in\N$ with $2 \leq k\leq n+1$
\[
 \sum_{\substack{r+s=n+1\\ i+j=k \\\text{$(s,r-1)$-unshuffles $\sigma$}}}  \epsilon(\sigma)\sgn(\sigma)(-1)^{r(s-1)} \leftsub{i}{l}_r(\leftsub{j}{l}_s(v_{\sigma(1)},\dotsc v_{\sigma(j)}),v_{\sigma(j+1)},\dotsc, v_{\sigma(n)}).
\]
Here the sign $\epsilon(\sigma)$ is the sign appearing from the Koszul-Quillen sign rule.
% \[
%  u_1\wedge\dotsb\wedge u_i = \epsilon(\sigma) \sgn(\sigma) u_{\sigma(1)}\wedge\dotsb\wedge u_{\sigma(i)}
% \]
\end{defn*}

\begin{rem*}
Notice that the subfamilies $\{\leftsub{1}{l}_i\}_{i\in\N}$ and $\{\leftsub{i}{l}_i\}_{i\in\N}$ both are $\ELL_\infty$ algebras sharing the same differential $\leftsub{1}{l}_1$. The rest of the brackets model the higher homotopies of the compatibility of the brackets $\leftsub{1}{l}_2$ and $\leftsub{2}{l}_2$. If these are the only non-zero brackets, then an $\ELL^2_\infty$-algebra is a pair of compatible Lie algebras.
\end{rem*}

\subsection{From dioperads to props}

There exists a forgetful functor from the category of properads to the category of dioperads which is denoted by $\Ucal^{\dioperad}_{\properad}$. It keeps the same underlying $\s$-bimodule but only allows composition along graphs of genus zero. The functor $\Ucal^{\dioperad}_{\properad}$ has a left adjoint which is denoted by $\Free^{\properad}_{\dioperad}$. For a quadratic dioperad $\Pcal=\Free^{\downarrow}_{c,0}(M)/(R)$ we have $\Free^{\properad}_{\dioperad}(\Pcal)=\Free^{\downarrow}_c(M)/(R)$, where in the latter case $(R)$ is the properadic ideal generated by $R$. The functor $\Free^{\properad}_{\dioperad}$ is not exact, Proposition 45 of \cite{Merkulov2007}, however in the same paper it is proved, Proposition 48, that if a dioperad is given by a distributive law then a quasi-free resolution of the dioperad is still a resolution when this functor is applied.

The step from properads to props is less troublesome. There exists a similar pair of functors $\Ucal^{\properad}_{\propp}$ and $\Free^{\propp}_{\properad}$. Also here it is true that for a quadratic properad $\Pcal=\Free^{\downarrow}_{c}(M)/(R)$ we have $\Free^{\propp}_{\properad}(\Pcal)=\Free^{\downarrow}(M)/(R)$, where $(R)$ is the propic ideal generated by $R$. By \S 7.4 of \cite{Vallette2007b} the functor $\Free^{\propp}_{\properad}$ is exact.  Let $\Free^{\propp}_{\dioperad}$ denote the composition $\Free^{\propp}_{\properad}\circ\Free^{\properad}_{\dioperad}$. We obtain the following result.

\begin{prop}
With the notation
\[
\Bipoisson=\Free^{\downarrow}_{c,0}(M)(R) \bhs\text{and}\bhs \Bipoisson_\infty=(\Free^{\downarrow}_{c,0}(E),\delta)
\]
we have
\[
 \Free^{\propp}_{\dioperad}(\Bipoisson)=\Free^{\downarrow}(M)(R) \bhs\text{and}\bhs\Free^{\propp}_{\dioperad}(\Bipoisson_\infty)=(\Free^{\downarrow}(E),\delta),
\]
moreover, the latter is a quasi-free resolution of the former.
\end{prop}

We will use the same notation for $\Bipoisson$ when considering it as prop.

\subsection{From props to wheeled props}
\label{propstowheeledprops}

There exist another pair of adjoint functors $\Ucal^{\propp}_{\wheeledpropp}$ and $\Free^{\wheeledpropp}_{\propp}$ between props and wheeled props. 
Unfortunately the latter functor is not exact; it has been shown that when applying $\Free^{\wheeledpropp}_{\propp}$ to the propic resolution of $\Poisson$, new cohomology classes arise, Remark 4.2.4 of  \cite{Merkulov2008}. In the same paper it was shown though, that a minimal quasi free wheeled propic resolution exists, Theorem 4.5.1, but neither the differential nor the $\sbimodule$ by which it is generated need necessarily be directly obtained from the propic resolution. The explicit calculation of the wheeled resolution is a highly non-trivial problem which has not yet been accomplished. Since $\Poisson_\infty$ is present in $\Bipoisson_\infty$ as a subcomplex, consider e.g.~all generators with only white operations, at least the same difficulties arise when trying to extend the propic resolution of $\Bipoisson$. 
\section{Prop profiles III: Geometrical interpretation}
\label{geominterpretation}

In this section we first translate representations of $\Bipoisson_\infty$ into polyvector fields. We then propose a definition of bi-Hamiltonian structures on formal graded manifolds. Finally we give an interpretation of such structures as a family of brackets comprising an $\ELL^2_\infty$ algebra on the structure sheaf of the manifold.

\subsection{Representations of $\Bipoisson_\infty$ as polyvector fields}
\label{repsofbipoisson}

A representation of $\Bipoisson_\infty$ in a dg vector space $(V,d)$ is a family of degree zero linear maps
\[
 \{ \leftsub{k}{\mu}^n_m\colon V^{\odot n}\to V^{\wedge m}[2-m]\}_{\substack{m,n\geq 1 \\ m+n\geq 3\\ 0\leq k \leq m-1}}
\]
satisfying certain quadratic relations. We set $\leftsub{0}{\mu}^1_1:=-d$ and note that $|\leftsub{0}{\mu}^1_1|=1=2-1$. We construct polyvector fields on the formal graded manifold associated to $V$ from a representation of $\Bipoisson_\infty$ as follows. For $0\leq k\leq i-1$ let
\[
 \leftsub{k}{\G}^i_{j}:=\frac{1}{i!j!}\leftsub{k}{\G}^{a_1\dotsb a_i}_{b_1 \dots b_j}t^{b_1}\dotsb t^{b_j}\psi_{a_1}\dotsb\psi_{a_i}.
\]
Here the elements $\leftsub{k}{\Gamma}^{a_1\dotsb a_i}_{b_1 \dots b_j}\in\K$ are given by
\begin{equation}
\label{gammadef}
 \leftsub{k}{\mu}^j_i(e_{b_1}\odot \dotsb \odot e_{b_j})=\leftsub{k}{\G}^{a_1\dotsb a_i}_{b_1 \dots b_j}e_{a_1}\wedge\dotsb\wedge e_{a_i}.
\end{equation}

To assemble these polyvector fields into a single entity we introduce a formal parameter $\hslash$ of degree zero; we define an element of $\G\in\pvfsh$ as follows
\[
\G:=\sum_{k\geq 0}\leftsub{k}{\G}\hslash^k, \bhs\bhs \text{where} \shs \leftsub{k}{\G}:=\sum_{\substack{i\geq k+1\\j \geq 1 %\\i+j\geq 2
}}\leftsub{k}{\G}^i_j.
\]
The role of the formal parameter $\hslash$ is to distinguish polyvector fields of the same weight from each other. Note that the part $\leftsub{k}{\G}$ comes from exactly those maps $\leftsub{k}{\mu}^n_m$ which are obtained from basis elements of $\Bipoisson_\infty$ with $k$ black operations, cf.~Theorem \ref{bipoissonresolutionthm}. We let $[\_,\_]_{\tilde{S}_{\hslash}}$ denote the  linearization in $\hslash$ of the Schouten bracket.

Note also that $\G$ satisfies $\leftsub{k}{\G}\in\wedge^{\bullet\geq k+1}\Tcal_V$. In fact it is easy to see that the elements with this property form a Lie subalgebra $\gfrak_V$ of $\pvfsh$.

Conversely, to an element $\G\in\gfrak_V$ one can by reversing the above process associate a family of maps $\{ \leftsub{k}{\mu}^n_m\}$. 

\begin{prop}
\label{bipoissonpolyvectorprop}
 A family of maps 
 \[
 \{\leftsub{k}{\mu}^n_m\colon V^{\odot n}\to V^{\wedge m}\}_{\substack{m,n\geq 1 \\ 0\leq k \leq m-1}}
 \]
 is a representation of $\Bipoisson_\infty$ in $V$ if and only if the corresponding polyvector field 
 \[
 \G=\sum_{k\geq 0}\leftsub{k}{\G}\hslash^k\in\gfrak_V
 \]
 satisfies the properties 
%$[\G,\G]_{\tilde{S}_\hslash}=0$, $|\G|=2$, and $\G|_*=0$.
\begin{enumerate}
\item $|\G|=2$,
\item $[\G,\G]_{\tilde{S}}=0$,
\item $\G|_0=0$.
\end{enumerate}
\end{prop}

\begin{proof}
Consider a representation $\rho\colon \Bipoisson_\infty \to \End_V$. Since the dioperad $\Bipoisson_\infty\newline=\Omega((\Bipoisson)^{\antishriek})$ is quasi-free, the differential $\delta$ is fully determined by the restriction to the weight one part, i.e.~graphs with one vertex. That $\rho$ is a representation of dg dioperads is thus equivalent to that the following diagram commutes for all $m,n\geq 1$ with $m+n\geq 3$
\[
\xymatrix{
\Sigma^{-1}\ol{\Bipoisson^{\antishriek}}(m,n) \ar[r]^{\rho} \ar[d]_{\delta}       & \Hom(V^{\otimes n}, V^{\otimes m}) \ar[d]^{d} \\
\Free_{(2)}(\Sigma^{-1}\ol{\Bipoisson^{\antishriek}})(m,n) \ar[r]^{\rho} &  \Hom(V^{\otimes n}, V^{\otimes m}).
}
\]
Depicting the differential $d$ by $\times$ and the image under $\rho$ of a decorated graph by the graph itself this is equivalent to that for all $m,n\in\N$ and all $0\leq i \leq m-1$
\begin{multline}
\label{corollaeqn}
\sum_{
\substack{1 \leq k \leq n \\ 0 \leq j \leq m-1 \\ i_1+i_2=i \\ (k,n-k)\text{-shuffles} \shs \tau \\ (j,m-j)\text{-shuffles} \shs \sigma }
}
\hspace{-15pt}(-1)^{\sgn(\sigma)+j(m-j)}\hspace{5pt}
\mirroredsplitboxcorolla{ \sigma(j+1) }{ \sigma(m) }{ \tau(k+1) }{ \tau(n)}{\sigma(1) }{ \sigma(j) }{ \tau(1) }{ \tau(k) }=\\
\sum_{(1,n-1)\text{-shuffles} \shs \tau }
\mirroredboxcorolladiffdown{1}{m}{i}{1}{n}+
\sum_{(m-1,1)\text{-shuffles} \shs \sigma }(-1)^{\sgn(\sigma)+(m-1)}
\mirroredboxcorolladiffup{1}{m}{i}{2}{n}.
\end{multline}
This condition translates into a sequence of quadratic relations on the family of maps $\{\leftsub{k}{\mu}^n_m\colon V^{\odot n}\to V^{\wedge m}\}$ corresponding to $\rho$.

Identifying the differential $d$ and the $\leftsub{k}{\mu}^n_m$ decorating the vertices with the appropriate $\leftsub{k}{\G}^m_n$ of \eqref{gammadef} we first note that $|\leftsub{k}{\G}^m_{n}|=2$ is equivalent to $|\leftsub{k}{\mu}^n_m|=2-m$. That $\G$ satisfies $\G|_0=0$ is immediate since $n\geq 1$ for the maps $\leftsub{k}{\mu}^n_m$. Now consider the expression $[\G,\G]_{S_{\hslash}}$. It is zero precisely when the coefficients of all monomials 
\[
 t^{b_1}\dots t^{b_{n}} \psi_{a_1} \dots \psi_{a_{m}} \hslash^{k}.
\]
in its expansion vanish. The condition $[\G,\G]_{S_{\hslash}}$=0 is thus equivalent to that for all $m,n\geq 1$, all $0\leq i\leq m-1$, all $a_1, \dotsc, a_{m}$, and all $b_1, \dotsc, b_{n}$ the following equality holds
\begin{equation}
\label{gammaeqn}
\sum_{\substack{
0\leq j \leq m-1 \\ 1 \leq k \leq n \\ i_1+i_2=i \\ 
(j,m-j)\text{-shuffles} \shs \sigma \\ (k,n-k)\text{-shuffles} \shs \tau
}}
  (-1)^{\sgn(\sigma)+j(m-j)}
 \frac{\leftsub{i_1}{\G}^{a_{\sigma(1)} \dots a_{\sigma(j)} e}_{b_{\tau(1)} \dots b_{\tau(k)}}}{j!k!}
 \frac{\leftsub{i_2}{\G}^{a_{\sigma(j+1)} \dots a_{\sigma(m)}}_{e b_{\tau(k+1)} \dots b_{\tau(n)}}}{(m-j)!(n-k)!}=0.
\end{equation}
It is straightforward to check that \eqref{gammaeqn} is satisfied if and only if \eqref{corollaeqn} is.
\end{proof}

Theorem \ref{bipoissoninV} is now an immediate consequence of the preceding proposition which also prompts us to make the following definition.

\begin{defn*}
An \emph{extended bi-Hamiltonian structure} on a formal graded manifold $V$ is an element $\G\in\gfrak_V$ of degree two satisfying $[\G,\G]_{\tilde{S}_\hslash}=0$.
\end{defn*}

With this definition Theorem \ref{bipoissoninV} can be reformulated as follows.

\begin{thm}
\label{mainthm}
There is a one-to-one correspondence between representations of $\Bipoisson_\infty$ in a dg vector space $V$ and pointed extended bi-Hamiltonian structures on the formal graded manifold associated to $V$.
\end{thm}

Regarding non-pointed bi-Hamiltonian structures cf.~Remark \ref{nonpointedrem}.

\subsection{A conceptual interpretation}
 To make the correspondence between representations of $\Bipoisson_\infty$ and polyvector fields in $\gfrak_V$ clearer we can use a result of Merkulov and Vallette. In \cite{Merkulov2007} they showed, Corollary 26, that the family $\Hom^{\s}(\Ccal,\Pcal)=\{\Hom^{\s}(\Ccal,\Pcal)(m,n)\}_{m,n\in\N}$ of all $\sbimodule$ graded homomorphisms from a coproperad $\Ccal$ to a properad $\Pcal$ is naturally a dg Lie algebra with
\[
 [\alpha,\beta]:=\mu_\Pcal \circ ((\alpha^{v_1},\beta^{v_2})-(-1)^{\alpha\beta}(\beta^{v_1},\alpha^{v_2}))\circ\Delta,
\]
where $v_1$ and $v_2$ are the lower and upper vertices, respectively, of the decorated graphs in the image of $\Delta$, and
\[
 \delta(\alpha):=\delta_\Pcal \circ \alpha - (-1)^{\alpha} \alpha \circ \delta_\Ccal.
\]
In the case when $\Ccal=\Sigma\ol{\Qcal^{\antishriek}}$, $\Qcal$ is Koszul, and $\Pcal=\End_V$, the Maurer-Cartan elements of $\Hom^{\s}(\Ccal,\Pcal)$, i.e.~the elements $\gamma$ which satisfy $\delta(\gamma)+\half[\gamma,\gamma]=0$, are precisely the representations of $\Qcal_\infty$ in $V$. 

With $\Ccal=\Bipoisson^{\antishriek}$ and $\Pcal=\End_V$ we have
\[
 \Hom^{\s}(\Ccal,\Pcal)(m,n)=\hspace{-4.9pt}\bigoplus_{\substack{m,n\geq 1\\ m+n\geq 3}}\hspace{-4.9pt} \underbrace{\odot^n V^* \otimes \wedge^m V[2-m]\oplus\dotsb\oplus \odot^n V^* \otimes \wedge^m V[2-m]}_{\text{$m$ copies}}\iso\tilde{\gfrak}_V.
\]
Here $\tilde{\gfrak}_{V}$ is the Lie subalgebra of $\gfrak_V$ consisting of all elements 
\[
\G=\sum_{\substack{i,j\geq 1\\0 \leq k \leq i-1}}\leftsub{k}{\G}^i_{j}:=\leftsub{k}{\G}^{a_1\dotsb a_i}_{b_1 \dots b_j}t^{b_1}\dotsb t^{b_j}\psi_{a_1}\dotsb\psi_{a_i}\hslash^k 
\]
such that the part $\leftsub{0}{\G}^{a_1}_{b_1} t^{b_1} \psi_{a_1}$ is zero, and the isomorphism is the one given in \S \ref{repsofbipoisson}. The differential $d$ of $V$ translates to a vector field $D=D^a_b t^b \psi_a$, for $d(e_b)=D^a_b e_a$, which in turns yields a differential $\delta_\hslash=[D,\_]_{\tilde{S}_\hslash}$ on $\tilde{\gfrak}_{V}$. Using Proposition \ref{bipoissonpolyvectorprop} it is not hard to show that the above isomorphism is an isomorphism of dg Lie algebras. From an element $\tilde{\G}\in\tilde{\gfrak}_V$ we obtain an element $\G=D+\tilde{\G}\in\gfrak_V$ and we have that $\tilde{\G}$ is a Maurer-Cartan element if and only if $[\G,\G]_{\tilde{S}_\hslash}=0$. 

With $\Ccal$ and $\Pcal$ as in the previous paragraph, the set of morphisms $\Hom^{\s}(\Ccal,\Pcal)$ is isomorphic to the underlying space of the deformation complex of $\Bipoisson$ algebras. Thus we see that the deformation theory of $\Bipoisson$ algebras is directly related to the Lie algebra $\gfrak_V$. See \cite{Merkulov2007} for more on the deformation complex.

\subsection{Representations of $\Bipoisson_\infty$ in non-graded vector spaces}

If the vector space $V$ is concentrated in degree zero then the maps $\leftsub{k}{\mu}^n_m$ corresponding to a representation of $\Bipoisson_\infty$ vanish unless $m=2$. Thus $\G=\leftsub{0}{\G}+\leftsub{1}{\G}\hslash$ and $\leftsub{0}{\G}$ and $\leftsub{1}{\G}$ are bivector fields. The condition $[\G,\G]_{\tilde{S}_{\hslash}}=0$ is therefore equivalent to
\[
 [\leftsub{0}{\G},\leftsub{0}{\G}]_{\tilde{S}}+
 ([\leftsub{0}{\G},\leftsub{1}{\G}]_{\tilde{S}}+[\leftsub{1}{\G},\leftsub{0}{\G}]_{\tilde{S}})\hslash+
 [\leftsub{1}{\G},\leftsub{1}{\G}]_{\tilde{S}}\hslash^2=0
\]
and we observe that representations of $\Bipoisson_\infty$ in $V$ are in one-to-one correspondence with classical bi-Hamiltonian structures on the formal manifold associated to $V$. In particular this proves Theorem \ref{bipoissoninRn}.

\subsection{The family of brackets of an extended bi-Hamiltonian structure}
\label{abiggerfamily}

To an element $\G=\sum_{k\geq 0}\leftsub{k}{\G}\hslash^k\in\gfrak_V$ with $\leftsub{k}{\G}=\sum_{i\geq k+1}\leftsub{k}{\G}_i$ and
\[
 \leftsub{k}{\G}_i:=\leftsub{k}{\G}^{a_1\dots a_{i}}(t) \psi_{a_1}\dotsb\psi_{a_i}
\]
we associate a family of brackets as follows.
For $1\leq k \leq n$ we define an $n$-ary bracket $\leftsub{k}{L}_n\colon\otimes^n\Ocal_V\to\Ocal_V$  by
\begin{align*}
 \leftsub{k}{L}_n(f_1,\dotsc,f_n)_i:&=\leftsub{k-1}{\G}_n d f_1 \wedge \dotsb \wedge d f_n\\
&=(-1)^{\epsilon}\leftsub{k-1}{\G}^{a_1\dots a_{n}}(t)(\p_{a_1}f_1)\dotsb(\p_{a_n}f_n).
\end{align*}
Here the sign $(-1)^\epsilon$ is given by
\[
 \epsilon= a_n(f_1+\dotsb+f_{n-1}+n-1)+(a_{n-1})(f_1+\dotsb+f_{n-2}+n-2)+\dotsb+a_2(f_1+1).
\]

\begin{thm}
 The brackets $\leftsub{k}{L}_n$ associated to a polyvector field $\G\in\gfrak_V$ as above satisfy the Leibniz property in each argument, i.e.
\begin{multline*}
 \leftsub{k}{L}_n(f_1,\dotsc, f_{j-1},gh ,f_j,\dotsc f_n)=\\(-1)g \leftsub{k}{L}_n(f_1,\dotsc, f_{j-1},h ,f_j,\dotsc f_n)+(-1)\leftsub{k}{L}_n(f_1,\dotsc, f_{j-1},g ,f_j,\dotsc f_n)h.
\end{multline*}
Moreover, the family of brackets  $\{\leftsub{k}{L}_n \}_{i\geq 1, 1\leq k \leq n}$ gives $\Ocal_V$ the structure of $\ELL^2_\infty$ algebra if and only if $\G$ is of degree two and satisfies $[\G,\G]_{\tilde{S}_\hslash}=0$.
\end{thm}
\begin{proof}
The proof is completely analogous to that of Theorem \ref{poissonLinftythm}.
\end{proof}

This leads to another definition of extended bi-Hamiltonian structures on formal graded manifolds, which by the preceding theorem is equivalent to the one we gave in \S \ref{repsofbipoisson}.

\begin{defn*}
An \emph{extended bi-Hamiltonian structure} on a formal graded manifold $V$ is an $\ELL^2_\infty$ algebra $\{\leftsub{k}{L}_n \}_{n\geq 1, 1\leq k \leq n}$ on $\Ocal_V$ such that the brackets $\leftsub{k}{L}_n$ have the Leibniz property in each argument.
\end{defn*}
\appendix

\section{Details on $\Gfrak^*$-algebras}

\subsection{Operads and $\Gfrak^{\downarrow_1}_c$-algebras}
\label{galgvsoperad}

An operad is often defined as the data
\[
(\Pcal=\{\Pcal(n)\}_{n\in\N}, \{\circ_i^{n_1,n_2}\}_{\substack{n_1,n_2\in\N \\ 1\leq i \leq n_1}},\bbone),
\]
where $\Pcal$ is an $\s$-module, $\bbone\in\Pcal(1)$, and the maps
\[
 \circ^{n_1,n_2}_i\colon \Pcal(n_1)\otimes\Pcal(n_2)\to \Pcal(n_1+n_2-1)
\]
satisfy certain associativity, $\s$-equivariance, and unit axioms, see e.g.~\cite{Markl2002}.

\begin{prop}
 The above definition of an operad is equivalent to the definition of a $\Gfrak^{\downarrow_1}_c$-algebra.
\end{prop}

\begin{proof}
 Let $(\Pcal,\mu,\eta)$ be a $\Gfrak^{\downarrow_1}_c$-algebra. We can give $\Pcal$ an operad structure of the above type as follows. Let $G\in\Gfrak^{\downarrow_1}_c$ be the  two-vertex graph depicted in Figure \ref{twovertexgraph}.
\begin{figure}[h]
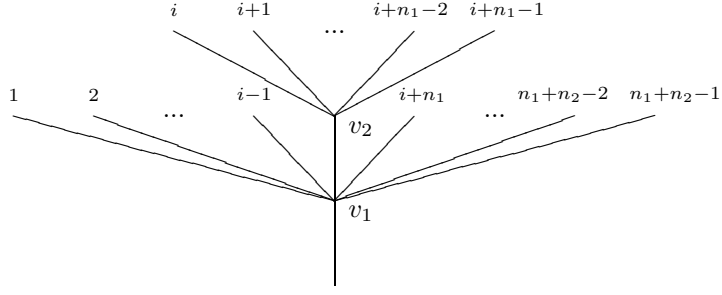

\[
 \twovertexgraph .
\]
\caption{\label{twovertexgraph} A two-vertex graph.}
\end{figure}
We then define $p_1 \circ_{i}^{n_1,n_2} p_2:=\mu_G(( p_1 \otimes_\s g_1)\otimes( p_2 \otimes_\s g_2))$, where $p_1$ and $p_2$ are decorating $v_1$ and $v_2$, respectively, and $g_1$ and $g_2$ are labelings satisfying
\[
 \begin{cases}
  g_1\circ\inp_G(1)=1,\dotsc, g_1\circ\inp_G(i-1)=i-1 \\
  g_2\circ\inp_G(i)=1,\dotsc, g_2\circ\inp_G(i+n_2-1)=n_2 \\
  g_1\circ\inp_G(i+n_2)=i+1,\dotsc, g_1\circ\inp_G(n_1+n_2-1)=n_1.
 \end{cases}
\]
The condition that $\mu_G=\mu_{G/H}\circ\mu^G_H$ for all pairs of a three vertex graph $G$ and a two vertex $\Gfrak^{\downarrow_1}_c$-admissible subgraph $H$ implies that the $\circ_i^{n_1,n_2}$ satisfy the associativity axioms of an operad. The  $\s$-equivariance axioms follow from the $\s$-equivariance of the $\mu_G$ as well as by the structure of decorated graphs. Defining $\bbone:=\eta(1)$, the unit axioms of an operad are immediate from those of the $\Gfrak^{\downarrow_1}_c$-algebra.

Conversely if an $\s$-bimodule has an operad structure then we can define a $\Gfrak^{\downarrow_1}_c$-algebra structure by letting $\mu_H$, for $H$ a two-vertex graph, be given by the appropriate $\circ_i^{n_1,n_2}$ as above. Then for a graph $G\in\Gfrak^{\downarrow_1}_c$, with $k=|V_G|>2$, we define
\[
\mu_G:=\mu_{(\dotsb(G/H_{1})/\dotsb/H_{k-1})}\circ\dotsb \circ \mu^{G/H_{1}}_{H_2} \circ \mu^G_{H_1},
\]
where $H_1,\dotsc H_{k-1}$ is an arbitrary sequence of two-vertex graphs such that $H_i$ is a $\Gfrak^{\downarrow_1}_c$-admissible subgraph of $(\dotsb(G/H_{1})/\dotsb/H_{i-1})$. That the maps $\mu_G$ are well-defined and satisfy $\mu_G=\mu_{G/H}\circ\mu^G_H$ is a consequence of the associativity and the $\s$-equivariance of the $\circ_i^{n_1,n_2}$. The unit is defined by $\eta(1):=\bbone$.
\end{proof}

Similarly, by considering appropriate two-level graphs, one can show that the definitions of $\Gfrak^\downarrow_{c,0}$-, $\Gfrak^\downarrow_c$-, and $\Gfrak^\downarrow$-algebras  correspond to the classical definitions of dioperads, properads, and props.

%%DIOPERAD SUBSECTION KEEP THIS
% \subsection{Dioperads and $\Gfrak^{\downarrow}_{c,0}$-algebras}
%
% Similarly as for operads, a dioperad can be described as a pair of collections
% \[
% \Pcal=(\{\Pcal(m,n)\}_{m,n\in\N}, \{\mbox{$_{\phantom{m_1,_1} i}^{m_1,n_1}\circ_j^{m_2,n_2}$}\}_{\substack{m_1,n_1,m_2,n_2\in\N \\ 1\leq i \leq n_1 \\ 1\leq j \leq m_2 }}),
% \]
% where $\Pcal=\{\Pcal(m,n)\}_{,n\in\N}$ is an $\s$-module and the maps
% \[
% \mbox{$_{\phantom{m_1,_1} i}^{m_1,n_1}\circ_j^{m_2,n_2}$}\colon \Pcal(m_1,n_1)\otimes\Pcal(m_2,n_2)\to \Pcal(m_1+m_2 -1,n_1+n_2 -1)
% \]
% satisfy certain associativity and $\s$-equivariance axioms. Just as for operads, the correspondence between dioperads and $\Gfrak^{\downarrow}_{c,0}$-algebras can be showed by associating to \mbox{$_{\phantom{m_1,_1} i}^{m_1,n_1}\circ_j^{m_2,n_2}$} the map $\mu_G$ where $G$ is:
% \[
%  \twovertexdioperadgraph .
% \]

\subsection{Composition product of free $\Gfrak^*$-algebras }
\label{freegrafting}
We keep the notation of \S \ref{freeGalgs}. When describing the grafting of graphs we will denote $G(G_1,\dotsc G_k)$ by $\widetilde{G}$.

The vertices of the graph $\widetilde{G}$ are given by $V_{\widetilde{G}}:=V_{G_1}\sqcup\dotsb\sqcup V_{G_k}$, the internal edges by $E^{\internal}_{\widetilde{G}}:= E^{\internal}_{G_1}\sqcup\dotsb\sqcup E^{\internal}_{G_k}\sqcup E^{\internal}_G$, and the external edges by $ E^{\inp}_{\widetilde{G}}:= E^{\inp}_{G}$ and $ E^{\outp}_{\widetilde{G}}:= E^{\outp}_{G}$. Defining the incidence morphism $\Phi_{\widetilde{G}}$ is more complicated.

For an edge $e\in E^{\internal}_{G_i}$ we define $\Phi_{\widetilde{G}}(e):=\Phi_{G_i}(e)$. Let $e\in E^{\internal}_{G}$ be an edge with $\Phi_G(e)=(v_i,v_j)$ and let the vertices $v_i$ and $v_j$ of $G$ be decorated with $f_i \otimes_{\s} G_i \otimes_{\s} g_i$ and $f_j \otimes_{\s} G_j \otimes_{\s} g_j$, respectively. Via the local labeling of $G$ and the global labelings of $G_i$ and $G_j$, this edge connects two vertices, $w_i\in V_{G_i}$ and $w_j\in V_{G_j}$  of $\widetilde{G}$, as follows. Let $e_i$ be the edge in $ E^{\outp}_{G_i}$ with $f_i \circ \outp_{G_i}(e_i)=e$. Note that this composition is well defined since another representative, $f'_i \otimes_{\s} G'_i \otimes_{\s} g'_i$ of the decoration of $v_i$, will satisfy $\outp_{G'_i}=\sigma\outp_{G_i}$ and $f_i'=f_i\sigma^{-1}$ for some permutation $\sigma$, implying $f_i'\circ\outp_{G'_i}=f_i\circ\sigma^{-1}\circ\sigma\circ \outp_{G_i}=f_i\circ\outp_{G_i}$. By composing further with $\inp_{G_j}\circ g_j$, which by a similar argument also is well defined, we obtain an edge $e_j=\inp_{G_j}\circ g_j \circ f_i \circ \outp_{G_i}(e_i)\in E^{\inp}_{G_j}$. Let $w_i=\Phi_{G_i}(e_i)$ and $w_j=\Phi_{G_j}(e_j)$, then we set $\Phi_{\widetilde{G}}(e):=(w_i,w_j).$

For an external edge $e\in E^{\inp}_{\widetilde{G}}$ with $\Phi_{G}(e)=v_i$ let  $e_i=\inp_{G_i}\circ g_i(e)\in E_{G_i}^{\inp}$ and $w_i=\Phi_{G_i}(e_i)$. We define $\Phi_{\widetilde{G}}(e):=w_i$. Similarly for an external edge $e\in E^{\outp}_{\widetilde{G}}$ with $\Phi_{G}(e)=v_i$ let  $e_i=f_i \circ \outp_{G_i}(e)\in E_{G_i}$ and $w_i=\Phi_{G_i}(e_i)$. We define $\Phi_{\widetilde{G}}(e):=w_i$. By the same arguments as above this is well defined. The global labeling of the external edges is directly induced by the one of $G$, $\inp_{\widetilde{G}}:=\inp_G$ and $\outp_{\widetilde{G}}:=\outp_G$.

For three edges $e, e_i, e_j$ connected as above we will use the notation $e_{\inp}:=e_i$, $e_{\outp}:=e_j$, and $(e_i)_{\connecting}=(e_j)_{\connecting}:=e$. We will use the same notation for two connected edges.

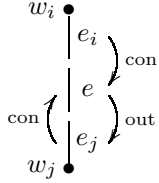
\[
\econinout
\]

The elements $\wtp^a_b$ are defined as follows. If $\olp^a_b= f^a_b \otimes_\s p^a_b \otimes_\s  g^a_b$  is an element decorating a vertex $w\in V_{G_i}$ with $|E^{\outp}_w|=m$ and $|E^{\inp}_w|=n$, then $\wtp^a_b= \widetilde{f}^a_b \otimes_\s p^a_b \otimes_\s  \widetilde{g}^a_b$, where the bijections $\widetilde{f}^a_b\colon [m] \to E^{\outp}_w$ and  $\widetilde{g}^a_b\colon E^{\inp}_w\to [n]$ are given by
\[
\widetilde{f}^a_b(i)=
\begin{cases}
f^a_b(j) & \hspace{-1.6pt}\text{if $f^a_b(j)\in E^{\internal}_{G_i}$ }\\
f^a_b(j)_{\connecting} & \hspace{-1.6pt}\text{if $f^a_b(j)\in E^{\outp}_{G_i}$}
\end{cases}
\shs \text{and} \shs
\widetilde{g}^a_b(e)=
\begin{cases}
g^a_b(e) & \hspace{-1.6pt}\text{if $e \in E^{\internal}_{G_i}$ }\\
g^a_b(e_{\outp}) & \hspace{-1.6pt}\text{if $e \in E_{G}\cap (E^{\inp}_w)_{\connecting} $}.
\end{cases}
\]

%\section*{Acknowledgements}
%The author is grateful to S.~Merkulov and J.~Gran\aa ker for valuable discussions.% on the contents of the paper.
%useful comments and suggestions

\bibliographystyle{abbrv}
\bibliography{compatiblepoisson}

\end{document}